\documentclass[11pt,reqno]{amsart}

\usepackage{amsmath, amsfonts, amsthm, amssymb, color, hyperref}

\newtheorem{Theorem}{Theorem}[section]
\newtheorem{Lemma}{Lemma}[section]
\newtheorem{Proposition}{Proposition}[section]

\theoremstyle{definition}

\theoremstyle{remark}
\newtheorem{Remark}{Remark}[section]

\numberwithin{equation}{section}

\allowdisplaybreaks

\renewcommand{\u}{{\bf u}}
\renewcommand{\H}{{\bf H}}
\newcommand{\R}{{\mathbb R}}
\newcommand{\C}{{\mathbb C}}
\newcommand{\Dv}{{\rm div}}

\def\c{\textrm{c}}

\newcommand{\HL}{{\mathcal L}}
\newcommand{\A}{{\mathcal A}}
\newcommand{\B}{{\mathcal B}}
\newcommand{\V}{{\mathcal V}}
\newcommand{\D}{{\mathcal D}}
\def\f{\frac}
\renewcommand{\O}{\Omega}
\renewcommand{\o}{\omega}

\def\hf1{^\f{1}{1-\xi^2}}

\def\be{\begin{equation}}
\def\en{\end{equation}}
\def\bs{\begin{split}}
\def\es{\end{split}}

\newcommand{\sgn}{\text{sgn}}

\newcommand{\F}{{\mathtt F}}

\newcommand{\M}{{\bf M}}

\newcommand{\p}{\partial}

\newcommand{\LV}{\left|}
\newcommand{\RV}{\right|}
\newcommand{\LN}{\left\|}
\newcommand{\RN}{\right\|}

\newcommand{\LC}{\left(}
\newcommand{\RC}{\right)}

\newcommand{\nc}{\textrm{nc}}
\newcommand{\n}{\textrm{n}}
\newcommand{\nn}{\textnormal{n}}

\newcommand{\WW}{\widehat{W}}
\newcommand{\Wphi}{\widehat{\varphi}}
\newcommand{\Wg}{\widehat{g}}
\newcommand{\vertiii}[1]{{\left\vert\kern-0.25ex\left\vert\kern-0.25ex\left\vert #1
    \right\vert\kern-0.25ex\right\vert\kern-0.25ex\right\vert}}
\newcommand{\ti}{i}

\author[R. M. Chen]{Robin Ming Chen}
\address{Department of Mathematics, University of Pittsburgh,
                           Pittsburgh, PA 15260.}
\email{mingchen@pitt.edu}

\author[J. Hu]{Jilong Hu}
\address{Department of Mathematics, University of Pittsburgh,
                           Pittsburgh, PA 15260.}
\email{jih62@pitt.edu}

\author[D. Wang]{Dehua Wang}
\address{Department of Mathematics, University of Pittsburgh,
                           Pittsburgh, PA 15260.}
\email{dwang@math.pitt.edu}

\title
[Stability of Vortex Sheets in Elastodynamics]
{Linear Stability of Compressible Vortex Sheets in Two-Dimensional Elastodynamics}


\date{\today}

\begin{document}

\begin{abstract}
The linear stability of rectilinear compressible vortex sheets is studied for two-dimensional isentropic elastic flows. This problem has a free boundary and the boundary is characteristic. A necessary and sufficient condition is obtained for the linear stability of the rectilinear vortex sheets. More precisely, it is shown that, besides the stable supersonic zone, the elasticity exerts an additional stable subsonic zone. 
A new feature for elastic flow is found that there is a class of states in the interior of subsonic zone where the stability of such states is weaker than the stability of other states in the sense that there is an extra loss of tangential derivatives with respect to the source terms. 
This is a new feature which Euler flows do not possess. 
One of the difficulties for the elastic flow is that the non-differentiable points of the eigenvalues may coincide with the roots of the Lopatinskii determinant. 
As a result, the Kreiss symmetrization cannot be applied directly.
Instead, we perform an upper triangularization of the system to separate only the outgoing modes at all points in the frequency space, so that 
an exact estimate of the outgoing modes can be obtained.
Moreover, all the outgoing modes are shown to be zero due to the $L^2$-regularity of solutions. The estimates for the incoming modes can be derived directly from the Lopatinskii determinant. 
This new approach avoids the lengthy computation and estimates for the outgoing modes when Kreiss symmetrization is applied. 
This method can also be applied to the Euler flows and MHD flows.  
\end{abstract}

\maketitle

\tableofcontents
\section{Introduction}

Vortex sheets are interfaces between two incompressible or compressible flows passing along each other. They arise in a broad range of physical problems in fluid mechanics, aerodynamics, oceanography and astrophysical plasma. Some typical examples include the sharp interface between two parallel shear flows, and vortex flows where the vortices are concentrated within a thin layer. 
In particular, for compressible flows, vortex sheets are fundamental waves which play an important role in the study of general entropy solutions to multi-dimensional hyperbolic systems of conservation laws. Analyzing the existence and stability of compressible vortex sheets may shed light on the 
understanding of fluid dynamics and the behavior of entropy solutions. 

In this paper, we are concerned with the vortex sheets in the following two-dimensional compressible inviscid flow in elastodynamics (\cite{dafermos2010hyperbolic, gurtin1981introduction, joseph1990fluid}):
\begin{align}
&\rho_{t}+\Dv(\rho\u)=0, \label{mass cons}\\
&(\rho\u)_{t}+\Dv(\rho\u\otimes\u)+\nabla p
=\Dv(\rho\F\F^\top),\label{momentum cons}\\
 \label{deformation cons}
&\F_t+\u\cdot\nabla\F=\nabla\u\F,
\end{align}
where $\rho$ stands for the density, $\u = (v,u) \in\R^2$ the velocity, $\F=(F_{ij})\in\M^{2\times2}$ the deformation gradient,
and $p$ the pressure where $p=p(\rho)$ is a smooth strictly increasing function in $(0,+\infty)$. 
This model arises from the viscoelastic fluids with negligible viscosity, and describes the features of both fluids and solids. 

Before going into further discussions of the elastodynamics model, we would like to first review some known results in the study of the compressible vortex sheets.
For vortex sheets in the compressible Euler flow, the study of stability in the linear regime dates back to 1950's by Miles \cite{miles1957reflection, miles1958disturbed} and Fejer-Miles \cite{fejer1963stability}. It is a classical result that, in two or three spatial dimensions with Mach number $M<\sqrt{2}$, the vortex sheets are violently unstable (c.f. \cite{miles1958disturbed, MR1775057}). These {instabilities} are the analogue of the Kelvin-Helmholtz instability for incompressible fluids. For the theory of the incompressible vortex sheets, we refer the readers to \cite{ambrose2007well, chemin1998perfect, jiu2006strong, jiu2007strong, lopes2001existence, majda2002vorticity, sulem1981finite, wang2012new, wu2006mathematical} and the references therein. In a series of papers, Artola-Majda \cite{artola1987nonlinear, artola1989nonlinear, artola1989nonlinearkink} investigated the interaction of the vortex sheets and highly oscillatory waves in the two-dimensional compressible Euler flow, indicating the global-in-time nonlinear instability of vortex sheets for Mach number $M>\sqrt{2}$. These instability results show that one can not expect the global existence of the vortex sheets in multi-dimensional spaces. In the pioneer works \cite{coulombel2004stability, coulombel2008nonlinear}, Coulombel and Secchi used a micro-local analysis and Nash-Moser method to establish the linear and local-in-time nonlinear stability in two dimensions with March number $M>\sqrt{2}$. In their setup, the initial data is a small perturbation of a rectilinear (i.e. piecewise constant) vortex sheet and their notion of linear stability is in a similar sense to that of shock waves by Majda \cite{majda1983stability, majda1983existence} and Coulombel \cite{coulombel2002weak, coulombel2004weakly}. Their results imply the local-in-time existence of such vortex sheets. 
For two-dimensional non-isentropic Euler flows, Morando-Trebeschi \cite{morando2008two} obtained a weaker linear stability of the vortex sheets, that is, there is some additional loss of derivatives in their estimates. Moreover, there have been a lot of studies on the vortex sheets in the multi-dimensional steady flows; see \cite{chen2013well, chen2007stability, wang2013stability, wang2014structural} and the references therein.

In complex fluids, the situation becomes more complicated.  
Ruan-Wang-Weng-Zhu \cite{ruanrectilinear} considered the linear stability of compressible vortex sheets in two-dimensional inviscid liquid-gas two-phase flows. They showed that the linear stability, compared with the Euler flow, may be weaker when the Mach number satisfies a certain condition. For the two-dimensional isentropic magnetohydrodynamics (MHD), Wang-Yu \cite{wang2013stabilization} showed that the magnetic fields will lower the critical Mach number and exert a small subsonic zone where some linear stability holds for the rectilinear current-vortex sheets. 
Moreover, for the three-dimensional compressible MHD, Blokhin-Trakhinin \cite{blokhin2002stability}, Trakhinin \cite{trakhinin2005existencevairable, trakhinin2005existencelinearized, trakhinin2009existence} and Chen-Wang \cite{chen2008existence, chen2012characteristic} adopted a different symmetrization approach to obtain certain linear and nonlinear stability, and the existence of the current-vortex sheets. The results of MHD indicate the stabilization effects of the magnetic fields on the current-vortex sheets.  

For the viscoelastic fluids, there have been extensive studies on various aspects from the modeling and analysis point of view \cite{dafermos2010hyperbolic, hu2010local, joseph2007potential, liu2001an, oldroyd1958non, renardy1987mathematical}, as well as on their applications \cite{goktekin2004method, kunisch2000optimal, yu2007two}. In the two important examples, namely, the shear flows and the vortex flows, both numerical experiments and theoretical analysis indicate that 
%
the viscoelasticity plays a stabilization role (see \cite{azaiez1994linear, huilgol2015fluid, kaffel2010stability, oldroyd1958non} and the references therein). Moreover, vortex sheets in viscoelastic fluids have also been discussed by Huilgol \cite{huilgol1981propagation, huilgol2015fluid}, where the author considered the Rayleigh problem in viscoelastic fluids and showed by constructing an example that the unsteady shearing motions can lead to vortex sheets in some viscoelastic liquids. On the other hand, Hu-Wang \cite{hu2012formation} showed the formation of singularity and the breakdown of classical solutions to system \eqref{mass cons}-\eqref{deformation cons} for certain initial data. These results motivate us to consider the inviscid elastic fluids and investigate the stabilization effects of the elasticity on the vortex sheets. Specifically, we obtain the necessary and sufficient conditions for the linear stability of rectilinear vortex sheets in a two-dimensional compressible isentropic inviscid elastic fluid in the sense of \cite{coulombel2004stability} by analyzing the Lopatinskii determinant of the linearized boundary value problem. 

As in the aforementioned works of Euler flow, two-phase flow and MHD flow, a common challenge of the vortex sheet problem is that the system has a free boundary, and the free boundary is characteristic.
 The characteristic boundary leads to some loss of control on the trace of the characteristic parts of the solutions \cite{coulombel2004stability, lax1960local, majda1975initial}. Moreover, 
the Kreiss-Lopatinskii condition does not hold uniformly, which implies some loss of the tangential derivatives in the estimates of the solutions in terms of the sources on the right side of the linearized problem \cite{coulombel2002weak, coulombel2004weakly, coulombel2004stability}.

In addition to the above difficulties, for the elastic flow, the appearance of the elasticity leads to a more complicated distribution of the roots in the Lopatinskii determinant. More precisely, the non-differentiable points of the eigenvalues may coincide with the roots of the Lopatinskii determinant, which does not happen for Euler flow (c.f. \cite{coulombel2004stability}). In the standard arguments \cite{coulombel2002weak, coulombel2004stability, kreiss1970initial}, one needs to construct the Kreiss symmetrizer and work along with the Lopatinskii determinant at each point in the frequency space. At the non-differentiable points of the eigenvalues, the usual Kreiss symmetrizer requires no loss of derivatives with respect to the source term; but the degeneracy at the roots of the Lopatinskii determinant implies the loss of the tangential derivatives, which leads to some serious obstacles to apply the Kreiss symmetrizer argument. 
%
To overcome this difficulty, we deal with this problem in a different way. More precisely, instead of using Kreiss symmetrization to separate every single mode, we first perform an upper triangularization of the system to separate only the outgoing modes from the system at all points in the frequency space. Then we establish an exact estimate on the outgoing modes from the equation, and use the $L^2$-regularity of solutions to conclude that the outgoing modes in the homogeneous system are zero. As a result, we only need to estimate the incoming modes, which can be derived directly from the Lopantiskii determinant. Therefore combining the estimates for the outgoing and incoming modes we achieve the linear stability. {In conclusion, we find that with the added elasticity, a new stability region can be generated in the subsonic zone, indicating the stabilization effect of elasticity as expected. We also find that within the stable subsonic region there exist a class of states where the stability of such states is weaker than the stability of other states in the sense that there is an extra loss of tangential derivatives, due to the fact that the Lopatinskii determinant exhibits higher order of degeneracy at such states. This is a new feature which Euler flows do not possess.} We further remark that our upper triangularization simply tells us that the outgoing modes are all zero. Thus we can avoid the lengthy computation and estimates for the outgoing modes when the Kreiss symmetrization is applied, and hence the arguments are greatly simplified. This upper triangularization method can also be applied to the Euler and MHD flows, as illustrated briefly in Section \ref{application}.


The result in this paper concerns with the linear stability with constant coefficients of the elastic vortex sheets. The natural next step is to consider the linear stability with variable coefficients, and more interestingly, the nonlinear stability of elastic vortex sheets, which would be a crucial tool for proving the local existence. These issues will be addressed in a forthcoming paper. 

The rest of the paper is organized as follows. In Section \ref{sec_formulation}, we formulate the nonlinear problem of the vortex sheets, fix the boundary, linearize the system around a given trivial solution, state our main theorem and sketch the idea of the proof. In Section \ref{sec_preliminary}, we perform some preliminary reduction and eliminate the front, which leads us to a system of ordinary differential equations. In Section \ref{sec_normal mode}, we consider the normal mode analysis and separate the modes. In Section \ref{sec_Lopatinskii}, we investigate the Lopatinskii determinant and introduce the estimates near the roots of the Lopatinksii determinant. In Section \ref{sec_energy estimate}, we combine the separation of the modes and the estimates from the Lopatinskii determinant to obtain the energy estimate, and hence finish the proof of our main theorem. In the last section, we apply our method to the vortex sheets in Euler and MHD flows.

\bigskip
\section{Formulation of the Problems and Main Results}\label{sec_formulation}

In this section, we derive the governing dynamics of the vortex sheets from the elastodynamical equations \eqref{mass cons}-\eqref{deformation cons}, linearize them around a rectilinear vortex sheet, and state our main theorem.

\subsection{Governing equations for vortex sheets}

To obtain the equations of the vortex sheet, we first reformulate the system \eqref{mass cons}-\eqref{deformation cons} into a divergence form.
Notice that the intrinsic property $\Dv(\rho\F^\top)=0$ holds at any time throughout the flow if it is satisfied initially \cite{hu2011global}. We can rewrite the system as the following form:
\begin{equation}\label{cons form}
\begin{cases}
\rho_{t}+\Dv(\rho\u)=0, \\
(\rho\u)_{t}+\Dv(\rho\u\otimes\u)+\nabla p
-\Dv(\rho\F\F^\top)=0,\\
(\rho\F_j)_t+\Dv(\rho\F_j\otimes\u-\u\otimes\rho\F_j)=0,
\end{cases}
\end{equation}
where $\F_j$ is the $j$th column of the deformation gradient $\F=(F_{ij})$, and $i,j=1,2$. In column-wise components, the intrinsic property actually means 
\begin{equation}\label{div free F}
\Dv(\rho\F_j)=0 \quad \text{for } j=1,2.
\end{equation}

Consider $U(t,x_1,x_2)=(\rho,\u,\F)(t,x_1,x_2)$ to be a solution to the system \eqref{cons form} which is smooth on each side of a smooth hypersurface $\Gamma=\{x_2=\psi(t,x_1)\}$. We denote by $\nu=(-\p_1\psi, 1)$ the unit normal vector on $\Gamma$ and
\begin{equation*}
  U(t,x_1,x_2)=\begin{cases}
               U^+(t,x_1,x_2),\quad\quad \text{when $x_2>\psi(t,x_1)$},\\
               U^-(t,x_1,x_2),\quad\quad \text{when $x_2<\psi(t,x_1)$},
            \end{cases}
\end{equation*}
where $U^\pm = (\rho^\pm,\u^\pm,\F^\pm)$. It follows that $U$ satisfies the Rankine-Hugoniot conditions at each point of $\Gamma$:
\begin{subequations}
\begin{align}
&\p_t\psi[\rho]-[\rho\u\cdot\nu]=0,\label{Jmass} \\
&\p_t\psi[\rho\u]-[(\rho\u\cdot\nu)\u]-[p]\nu+[\rho\F\F^\top\nu]=0,\label{Jmomentun}\\
&\p_t\psi[\rho\F_j]-[(\u\cdot\nu)\rho\F_j]+[(\rho\F_j\cdot\nu)\u]=0,\label{Jdeformation}\\
&[\rho\F_j\cdot\nu]=0, \label{JrhoF}
\end{align}
\end{subequations}
where $[f]$ denotes the jump of the function $f$ across the hypersuface $\Gamma$. Here we have used \eqref{div free F} to derive the last jump condition.
Now, we denote $m_\nu=\rho(\psi_t-\u\cdot\nu)$. By (\ref{Jmass}), we have $[m_\nu]=0$. Thus combining (\ref{JrhoF}) and (\ref{Jdeformation}), we can obtain
\begin{align}\label{Jdeformation'}
&m_\nu[\F_j]+(\rho\F_j\cdot\nu)[\u]=0.
\end{align}
Since we consider the contact discontinuity, we assume $\p_t\psi=\u^+\cdot\nu=\u^-\cdot\nu$. By (\ref{Jmass}) and (\ref{Jdeformation'}), we obtain $[p]=0$, $m_\nu=0$, $[\F_j\cdot\nu]=0$ and $\F_j\cdot\nu=0$, for $j=1,2$. Therefore for a vortex sheet in the elastic flow, the jump conditions \eqref{Jmass}-\eqref{JrhoF} become
\begin{equation}\label{rankinehugoniot}
\rho^+=\rho^-, \quad \psi_t=\u^+\cdot\nu=\u^-\cdot\nu \text{ and } \F_j^+\cdot\nu=\F_j^-\cdot\nu=0, \quad j=1,2 \quad \text{on } \Gamma.
\end{equation}

We now introduce the transformations $\Phi^\pm(t,x_1,x_2)$ to straighten the free boundary $\Gamma$ as follows.
We first consider the class of functions $\Phi(t,x_1,x_2)$ such that  $\p_{2}\Phi(t,x_1,x_2)\geq \kappa >0$ and $\Phi(t,x_1,0)=\psi(t,x_1)$. Then we define
\begin{align*}
U^{\pm}_\#=(\rho_\#^\pm,\u_\#^\pm,\F_\#^\pm)(t,x_1,x_2)=(\rho,\u,\F)(t,x_1,\Phi(t,x_1,\pm x_2))
\end{align*}
for $x_2\geq0$. For simplicity of notation, we drop the $\#$ index in the rest of the paper. Define  $\Phi^\pm(t,x_1,x_2)= \Phi(t,x_1,\pm x_2)$. To obtain a simpler formulation of the problem, we can choose $\Phi$ satisfying
\begin{align*}
\p_t\Phi^\pm+v^\pm\p_1\Phi^\pm-u^\pm=0, 
\end{align*}
for any $x_2\geq0$, which is inspired by \cite{coulombel2004stability, francheteau2000existence}.
With this change of variables, we can rewrite the equations \eqref{mass cons}-\eqref{deformation cons} in the following form:
\begin{align}\label{differentialequation}
\p_tU^\pm+A_1(U^\pm)\p_{1}U^\pm+\frac{1}{\p_{2}\Phi^\pm}\left[A_2(U^\pm)-\p_t\Phi^\pm I-\p_{1}\Phi^\pm A_3(U^\pm)\right]\p_{2}U^\pm=0
\end{align}
for $x_2>0$ with fixed boundary $x_2=0$, where
\begin{align*}
A_1(U) =A_3(U)=
 \begin{pmatrix}
  v & \rho & 0 & 0 & 0& 0 & 0 \\
  \frac{p'}{\rho} & v & 0 & -F_{11} & 0 & -F_{12} & 0 \\
  0 & 0 & v & 0 & -F_{11} & 0 & -F_{12} \\
  0 & -F_{11} & 0 & v & 0 & 0 & 0 \\
  0 & 0 & -F_{11} & 0 & v & 0 & 0 \\
  0 & -F_{12} & 0 & 0 & 0 & v & 0 \\
  0 & 0 & -F_{12} & 0 & 0 & 0 & v \\
 \end{pmatrix}
\end{align*}
and
\begin{align*}
A_2(U) =
 \begin{pmatrix}
  u & 0 & \rho & 0 & 0& 0 & 0 \\
  0 & u & 0 & -F_{21} & 0 & -F_{22} & 0 \\
  \frac{p'}{\rho} & 0 & u & 0 & -F_{21} & 0 & -F_{22} \\
  0 & -F_{21} & 0 & u & 0 & 0 & 0 \\
  0 & 0 & -F_{21} & 0 & u & 0 & 0 \\
  0 & -F_{22} & 0 & 0 & 0 & u & 0 \\
  0 & 0 & -F_{22} & 0 & 0 & 0 & u \\
 \end{pmatrix}.
\end{align*}
Moreover, from \eqref{rankinehugoniot}, we obtain the boundary condition
\begin{align}\label{boundarycondition}
\begin{cases}
(v^+-v^-)\p_1\psi-(u^+-u^-)=0 \\
\p_t\psi+v^+\p_1\psi-u^+=0 \\
(F_{11}^+-F_{11}^-)\p_1\psi-(F_{21}^+-F_{21}^-)=0 \\
F_{11}^+\p_1\psi-F_{21}^+=0 \\
(F_{12}^+-F_{12}^-)\p_1\psi-(F_{22}^+-F_{22}^-)=0 \\
F_{12}^+\p_1\psi-F_{22}^+=0 \\
\rho^+-\rho^-=0
\end{cases}\quad, \quad \text{at } x_2=0,
\end{align}
where $\Phi^\pm=\psi$ at $x_2=0$.

It is easy to see that the above system \eqref{differentialequation}-\eqref{boundarycondition} admits rectilinear solutions. Moreover, all the rectilinear solutions can be transformed under the Galilean transformation and change of the scale of measurement to the following form: 
\begin{align}\label{specialsolution}
\mathring U^+ :=
\begin{pmatrix}
\mathring\rho \\
v^r \\
0 \\
F_{11}^r \\
0 \\
F_{12}^r \\
0
\end{pmatrix}, \quad
\mathring U^-  :=
\begin{pmatrix}
\mathring\rho \\
v^l \\
0 \\
F_{11}^l \\
0 \\
F_{12}^l \\
0
\end{pmatrix},\quad
\mathring\Phi^{\pm}(t,x_1,x_2)=\pm x_2,
\end{align}
where the constants $\mathring\rho$, $v^r$, $v^l$, $F_{11}^r$, $F_{11}^l$, $F_{12}^r$ and $F_{12}^l$ satisfy
\begin{align*}
v^r+v^l=F_{11}^r+F_{11}^l=F_{12}^r+F_{12}^l=0 \text{ and $v^r > 0$, $F_{11}^r$, $F_{12}^r\neq 0$}.
\end{align*}


\subsection{Linearized problem}

Now we linearize the system \eqref{differentialequation}-\eqref{boundarycondition} around the above constant states \eqref{specialsolution}. Denote $\dot{U}^{\pm}=(\dot{\rho}^{\pm},\dot{\u}^{\pm},\dot{\F}^{\pm})=U^{\pm}-\mathring U^{\pm}$ and $\dot{\Phi}^{\pm}=\Phi^{\pm}-\mathring\Phi^{\pm}$ to be the small perturbations of the constant solution, and consider the following linearized problem:
\begin{align*}
\p_t\dot{U}^\pm+A_1(\mathring U^\pm)\p_1\dot{U}^\pm\pm A_2(\mathring U^\pm)\p_2\dot{U}^\pm=0, 
\end{align*}
in ${x_2>0}$, with the boundary condition at $x_2=0$: 
\begin{align*}
\begin{cases}
(v^r-v^l)\p_1\varphi-(\dot{u}^+-\dot{u}^-)=0, \\
\p_t\varphi+v^r\p_1\varphi-\dot{u}^+=0, \\
(F_{11}^r-F_{11}^l)\p_1\varphi-(\dot{F}_{21}^+-\dot{F}_{21}^-)=0, \\
F_{11}^r\p_1\varphi-\dot{F}_{21}^+=0, \\
(F_{12}^r-F_{12}^l)\p_1\varphi-(\dot{F}_{22}^+-\dot{F}_{22}^-)=0, \\
F_{12}^r\p_1\varphi-\dot{F}_{22}^+=0, \\
(\dot{\rho}^+-\dot{\rho}^-)=0,
\end{cases}
\end{align*}
where $\dot{\Phi}^\pm=\varphi$ at $x_2=0$. In short, we have
\begin{equation}\label{system in short}
\begin{cases}L'\dot{U}=0, &\text{if $x_2>0$},\\
B(\dot{U},\varphi)=0, &\text{if $x_2=0$},
\end{cases}
\end{equation}
where
{\small
\begin{align*}
&L'\dot{U}=\p_t
\begin{pmatrix}
\dot{U}^+\\
\dot{U}^-
\end{pmatrix}+
\begin{pmatrix}
A_1(\mathring U^+) & 0\\
0             & A_1(\mathring U^-)
\end{pmatrix}
\p_1\begin{pmatrix}
\dot{U}^+\\
\dot{U}^-
\end{pmatrix}+
\begin{pmatrix}
A_2(\mathring U^+) & 0\\
0             & -A_2(\mathring U^-)
\end{pmatrix}
\p_2\begin{pmatrix}
\dot{U}^+\\
\dot{U}^-
\end{pmatrix},\\
&B(\dot{U},\varphi)=
\begin{pmatrix}
(v^r-v^l)\p_1\varphi-(\dot{u}^+-\dot{u}^-) \\
\p_t\varphi+v^r\p_1\varphi-\dot{u}^+ \\
(F_{11}^r-F_{11}^l)\p_1\varphi-(\dot{F}_{21}^+-\dot{F}_{21}^-) \\
F_{11}^r\p_1\varphi-\dot{F}_{21}^+ \\
(F_{12}^r-F_{12}^l)\p_1\varphi-(\dot{F}_{22}^+-\dot{F}_{22}^-) \\
F_{12}^r\p_1\varphi-\dot{F}_{22}^+ \\
(\dot{\rho}^+-\dot{\rho}^-)
\end{pmatrix}.
\end{align*}}

Next we symmetrize the system \eqref{system in short}. 
Consider the following change of variables,
\begin{equation}\label{defn Wnc}
W=
\begin{pmatrix}
 T & 0\\
 0 & T
\end{pmatrix}
\begin{pmatrix}
\dot{U}^+\\
\dot{U}^-
\end{pmatrix},
\end{equation}
where $T$ is a matrix of the following form:
\begin{align*}
T=\begin{pmatrix}
0 & 1 & 0 & 0 & 0 & 0 & 0 \\
-\frac{1}{2\mathring\rho} & 0 & \frac{1}{2c} & 0 & 0 & 0 & 0 \\
\frac{1}{2\mathring\rho} & 0 & \frac{1}{2c} & 0 & 0 & 0 & 0 \\
0 & 0 & 0 & 1 & 0 & 0 & 0 \\
0 & 0 & 0 & 0 & 1 & 0 & 0 \\
0 & 0 & 0 & 0 & 0 & 1 & 0 \\
0 & 0 & 0 & 0 & 0 & 0 & 1
\end{pmatrix},
\end{align*}
and $c=\sqrt{p'(\mathring\rho)}$ is the sound speed of the constant solution. For convenience, we denote the components of the new variable by
\begin{equation*}
W=(W_1,W_2,W_3,...,W_{14})^\top,
\end{equation*}
and the tangential, normal, characteristic and non-characteristic parts of $W$ by
\begin{equation}\label{defn W}
\begin{split}
&W^{\tan}=(W_1,W_4,W_6,W_8,W_{11},W_{13})^\top,\\
&W^\n=(W_2,W_3,W_5,W_7,W_9,W_{10},W_{12},W_{14})^\top,\\
&W^{\c}=(W_1,W_4,W_5,W_6,W_7,W_8,W_{11},W_{12},W_{13},W_{14})^\top,\\
&W^{\nc}=(W_2,W_3,W_9,W_{10})^\top.
\end{split}
\end{equation}

With the above change of variables, we multiply the system \eqref{system in short} by a symmetrizer
\begin{equation*}
\A_0=\text{diag}\{1,2c^2,2c^2,1,1,1,1,1,2c^2,2c^2,1,1,1,1\}.
\end{equation*}
Thus we can obtain the following equation 
\begin{equation}\label{symmetricsystem}
\begin{cases}
\A_0L'\dot{U} & = \HL W = \A_0\p_tW+\A_1\p_1W+\A_2\p_2W=0 ,\\
\displaystyle B(\dot{U},\varphi) & = \B (W^{\n}, \varphi)=\underline{M}W^{\n}+\underline{b}
\begin{pmatrix}
\p_t\varphi\\
\p_1\varphi
\end{pmatrix}=0,
\end{cases}
\end{equation}
where
\begin{align*}
&\A_1=
\begin{pmatrix}
\A_1^r & 0\\
0 & \A_1^l
\end{pmatrix} \quad \text{with}\\
&\A_1^{r,l}=\begin{pmatrix}
v^{r,l}                   & -c^2                & c^2               & -F_{11}^{r,l} & 0                    & -F_{12}^{r,l} & 0 \\
-c^2                     & 2c^2v^{r,l}      & 0                   & 0                  & -cF_{11}^{r,l} & 0          & -cF_{12}^{r,l} \\
c^2                      & 0                    & 2c^2v^{r,l}     & 0                  & -cF_{11}^{r,l} & 0           & -cF_{12}^{r,l} \\
-F_{11}^{r,l}         & 0                    & 0                   & v^{r,l}            & 0                   & 0           & 0 \\
0                          & -cF_{11}^{r,l} & -cF_{11}^{r,l} & 0                   & v^{r,l}            & 0           & 0 \\
-F_{12}^{r,l}         & 0                    & 0                   & 0                  & 0                    & v^{r,l}    & 0 \\
0                          & -cF_{12}^{r,l} & -cF_{12}^{r,l} & 0                  & 0                    & 0           & v^{r,l} \\
\end{pmatrix},\\\nonumber\\
&\A_2=\text{diag}\{0,-2c^3,2c^3,0,0,0,0,0,2c^3,-2c^3,0,0,0,0\},\\\nonumber\\
&\underline{M}=\begin{pmatrix}
-c & -c & 0 & 0 & c & c & 0 & 0\\
-c & -c & 0 & 0 & 0 & 0 & 0 & 0\\
-1 & 1 & 0 & 0 & 1 & -1 & 0 & 0\\
0 & 0 & -1 & 0 & 0 & 0 & 1 & 0\\
0 & 0 & -1 & 0 & 0 & 0 & 0 & 0\\
0 & 0 & 0 & -1 & 0 & 0 & 0 & 1\\
0 & 0 & 0 & -1 & 0 & 0 & 0 & 0
\end{pmatrix},\quad
\underline{b}=\begin{pmatrix}
0 & 2v^r\\
1 & v^r\\
0 & 0\\
0 & 2F_{11}^r\\
0 & F_{11}^r\\
0 & 2F_{12}^r\\
0 & F_{12}^r
\end{pmatrix}.
\end{align*}

We now introduce some weighted Sobolev spaces needed for the main theorem. Define
\begin{align*}
& H^s_{\gamma}(\R^2):=\{u(t,x_1)\in\D'(\R^2)\;:\;e^{-\gamma t}u(t,x_1)\in H^s(\R^2)\},\\
& H^s_{\gamma}(\R^3_+):=\{v(t,x_1,x_2)\in\D'(\R^3_+)\;:\;e^{-\gamma t}v(t,x_1,x_2)\in H^s(\R^3_+)\}
\end{align*}
for $s\in \R,\ \gamma\geq 1$, equipped with the norms
\begin{equation*}
\|u\|_{H^s_{\gamma}(\R^2)}:=\|e^{-\gamma t}u\|_{H^s(\R^2)}, \quad \|v\|_{H^s_{\gamma}(\R^3_+)}:=\|e^{-\gamma t}v\|_{H^s(\R^3_+)}
\end{equation*}
respectively, where
\begin{equation}
\R^3_+=\{(t,x_1,x_2)\in\R^3:x_2> 0\}.
\end{equation}
In our stability analysis, the norm that is frequently used is
\begin{equation*}
\|u\|_{s,\gamma}^2=\frac{1}{(2\pi)^2}\int_{\R^2}(\gamma^2+|\xi|^2)^s|\hat{u}(\xi)|^2d\xi,\quad \forall u\in H^s(\R^2),
\end{equation*}
where $\hat{u}(\xi)$ is the Fourier transform of $u$. Note that by letting $\tilde{u}:=e^{-\gamma t}u$, we have that $\|u\|_{H^s_{\gamma}(\R^2)}$ and $\|\tilde{u}\|_{s,\gamma}$ are equivalent, denoted by $\|u\|_{H^s_{\gamma}(\R^2)}\simeq\|\tilde{u}\|_{s,\gamma}$. 

Now we can define the space $L^2(\R_+;H^s_\gamma(\R^2))$, equipped with the norm
\begin{equation*}
\vertiii{v}^2_{L^2(H^s_\gamma)}=\int_0^{+\infty}\|v(\cdot,x_2)\|^2_{H^s_\gamma(\R^2)}dx_2.
\end{equation*}
Again we have 
\begin{equation*}
\vertiii{v}_{L^2(H^s_\gamma)}^2\simeq\vertiii{\tilde{v}}_{s,\gamma}^2:=\int_0^{+\infty}\|\tilde{v}(\cdot,x_2)\|^2_{s,\gamma}dx_2.
\end{equation*}
Note that $\|\cdot\|_{0,\gamma}$ is actually the usual norm on $L^2(\R^2)$ and $\vertiii{\cdot}_{0,\gamma}$ is the usual norm on $L^2(\R^3_+)$.


\subsection{Main result and methodology}
With the above notations and spaces, our stability result can be stated as follows.
\begin{Theorem}\label{main thm}
(1) If the particular solution defined by \eqref{specialsolution} satisfies
\begin{equation}\label{case1}
\begin{split}
&(v^r)^2> 2c^2+(F_{11}^r)^2+(F_{12}^r)^2,\quad \text{ or }\\
 &(v^r)^2< (F_{11}^r)^2+(F_{12}^r)^2, \text{ but } (v^r)^2\neq\frac{\left((F_{11}^r)^2+(F_{12}^r)^2\right)\left(2c^2+(F_{11}^r)^2+(F_{12}^r)^2\right)}{4\left( (F_{11}^r)^2+(F_{12}^r)^2+c^2\right)},
\end{split}
\end{equation}
then there is a positive constant $C$ such that for all { $\gamma\geq1$}, $W\in H^2_\gamma(\R^3_+)$ and $\varphi\in H^2_\gamma(\R^2)$, the following estimate holds:
\begin{equation}\label{estimate1}
\begin{split}
&\gamma\vertiii {W}^2_{L^2(H^0_\gamma)}+\|W^{{\textnormal{n}}}|_{x_2=0}\|^2_{L^2_\gamma(\R^2)}+\|\varphi\|^2_{H^1_\gamma(\R^2)}\\
&\quad\quad\leq C\left(\frac{1}{\gamma^3}\vertiii{\HL W}^2_{L^2(H^1_{\gamma})}+\frac{1}{\gamma^2}\|\B (W^{\textnormal{n}}|_{x_2=0}, \varphi)\|^2_{H^1_{\gamma}(\R^2)}\right).
\end{split}
\end{equation}

(2) If the particular solution defined by \eqref{specialsolution} satisfies
\begin{equation}\label{case2}
\begin{split}
&(v^r)^2=\frac{\left((F_{11}^r)^2+(F_{12}^r)^2\right)\left(2c^2+(F_{11}^r)^2+(F_{12}^r)^2\right)}{4\left( (F_{11}^r)^2+(F_{12}^r)^2+c^2\right)}, \quad\text{ or }\\
&(v^r)^2=(F_{11}^r)^2+(F_{12}^r)^2,
\end{split}
\end{equation}
then there is a positive constant $C$ such that for all { $\gamma\geq1$}, $W\in H^3_\gamma(\R^3_+)$ and $\varphi\in H^3_\gamma(\R^2)$, the following estimate holds:
\begin{equation}\label{estimate2}
\begin{split}
&\gamma\vertiii {W}^2_{L^2(H^0_\gamma)}+\|W^{\nn}|_{x_2=0}\|^2_{L^2_\gamma(\R^2)}+\|\varphi\|^2_{H^1_\gamma(\R^2)}\\
&\quad\quad\leq C\left(\frac{1}{\gamma^5}\vertiii{\HL W}^2_{L^2(H^2_{\gamma})}+\frac{1}{\gamma^4}\|\B (W^{\nn}|_{x_2=0}, \varphi)\|^2_{H^2_{\gamma}(\R^2)}\right).
\end{split}
\end{equation}

(3) If the particular solution defined by \eqref{specialsolution} satisfies
\begin{equation}\label{case3}
(v^r)^2=(F_{11}^r)^2+(F_{12}^r)^2+2c^2,
\end{equation}
then there is a positive constant $C$ such that for all { $\gamma\geq1$}, $W\in H^4_\gamma(\R^3_+)$ and $\varphi\in H^4_\gamma(\R^2)$, the following estimate holds: 
\begin{equation}\label{estimate3}
\begin{split}
&\gamma\vertiii {W}^2_{L^2(H^0_\gamma)}+\|W^{\nn}|_{x_2=0}\|^2_{L^2_\gamma(\R^2)}+\|\varphi\|^2_{H^1_\gamma(\R^2)}\\
&\quad\quad\leq C\left(\frac{1}{\gamma^7}\vertiii{\HL W}^2_{L^2(H^3_{\gamma})}+\frac{1}{\gamma^6}\|\B (W^{\nn}|_{x_2=0}, \varphi)\|^2_{H^3_{\gamma}(\R^2)}\right).
\end{split}
\end{equation}

(4) If the particular solution defined by \eqref{specialsolution} satisfies
\begin{equation}
(F_{11}^r)^2+(F_{12}^r)^2 < (v^r)^2 < 2c^2+(F_{11}^r)^2+(F_{12}^r)^2,
\end{equation}
the constant vortex sheets \eqref{specialsolution} is linearly unstable, in the sense that the Lopatinskii condition is violated.
\end{Theorem}
\begin{Remark}  The above theorem gives all the situations $(1)$-$(3)$ where the linear stability holds. On the other hand, for the remaining situation $(4)$, we can show that the linearized problem is unstable due to the failure of the Lopatinskii condition. The details will be discussed in Section \ref{sec_Lopatinskii}. Therefore this theorem actually gives us a sufficient and necessary condition on the stability of the linearized problem.
\end{Remark}
\begin{Remark}
(i) From the second condition of \eqref{case1} we observe that the elasticity term $\F$ produces a stable subsonic region, and hence provides a stabilization effect on the vortex sheets. It is easily seen that in absence of the elasticity $\F \equiv 0$, our results recover the stability theory for the vortex sheets in the Euler flow.

(ii) Another feature of the elastic flow which is different from the case in the Euler flow, can be inferred from the first condition in \eqref{case2}.  It follows from \eqref{case2} that there is a class of states in the interior of the subsonic region where the stability holds in a weaker sense, i.e. there is a two-order loss of derivatives in the estimates \eqref{estimate2}. 

(iii) The second condition in \eqref{case2} and the condition \eqref{case3} give the transition states across which the stability property changes. This property is similar to the Euler and MHD flows (c.f. \cite{coulombel2004on, wang2013stabilization}). At these two classes of states the stability holds in a weaker sense.
\end{Remark}

\medskip

Now we sketch the idea of the proof of the main theorem. We follow the standard argument (c.f. \cite{coulombel2004stability}) to first remove the source term from the equations, eliminate the wave front $\varphi$ from the resulting system, and then single out the non-characteristic part $W^{\nc}$ of the unknown $W$ to arrive at a reduced system in the Fourier space of the form:
\begin{equation}\label{idea}
\begin{cases}
\frac{d}{d x_2}\WW^{\nc}=A\WW^{\nc},\\\\ 
\beta\WW^{\nc}|_{x_2=0}=h,
\end{cases}
\end{equation}
where $A$ is a $4\times4$ block diagonal matrix and $\beta$ is a $2\times 4$ matrix (the explicit forms are given in \eqref{matrix A} and \eqref{matrix beta}). It turns out that all the desired estimates in Theorem \ref{main thm} can be achieved by estimating $\WW^\nc|_{x_2 = 0}$.  

By counting the number of boundary conditions one can only hope to control at most two components of $\WW^{\nc}|_{x_2=0}$. Thus in order to obtain the full estimate on $\WW^{\nc}|_{x_2=0}$ one has to utilize the differential equation in \eqref{idea}. The conventional way to do so is to use the Kreiss symmetrization. Roughly speaking, this argument is to first find a $4\times4$ Hermitian matrix $r$ and a number $C>0$ such that $r+C\beta^*\beta$ and $rA$ are both positive definite. The positivity of $r+C\beta^*\beta$ implies
\begin{equation}
<\WW^{\nc}, r\WW^{\nc}>|_{x_2=0}+C|h|^2\geq \kappa \left|\WW^{\nc}|_{x_2=0}\right|^2,
\end{equation} 
for some $\kappa>0$, where $<\cdot,\cdot>$ is the Euclidean inner product in $\C^4$. Therefore if one can further choose $r$ such that $<\WW^{\nc}, r\WW^{\nc}>|_{x_2=0}<0$, then $\WW^\nc|_{x_2 = 0}$ can be controlled by $h$. To check $<\WW^{\nc}, r\WW^{\nc}>|_{x_2=0}<0$, we multiply the differential equations in \eqref{idea} by $r^*\WW^{\nc}$ ($r^*$ is the Hermitian transpose of $r$) and integrate with respect to $x_2$ from $0$ to $\infty$ to obtain
\begin{equation}
-\frac{1}{2}<\WW^{\nc}, r\WW^{\nc}>|_{x_2=0}=\int_0^\infty<\WW^{\nc}, rA\WW^{\nc}>dx_2.
\end{equation}
Thus a sufficient condition for stability is that $rA$ is positive definite. However at the non-differetiable point of the eigenvalues of $A$, the usual symmetrizer $r$ can not make $r+C\beta^*\beta$ positive definite for any $C>0$ when the Lopatinskii determinant is zero there, and hence the Kreiss symmetrization method seems hard to apply. 
It turns out that this situation can happen for the vortex sheets in elastodynamics. More precisely, if the particular solution defined by \eqref{specialsolution} satisfies 
\begin{equation}\label{critical}
(v^r)^2=\frac{1}{4}\left((F_{11}^r)^2+(F_{12}^r)^2+c^2\right)
\end{equation}
with $c\leq \sqrt{3\left((F_{11}^r)^2+(F_{12}^r)^2\right)}$, some non-differentiable points of $A$ will coincide with some roots of the Lopatinskii determinant (c.f. Remark \ref{remark coincide}), and hence the Kreiss symmetrization cannot be applied directly. We will develop some new idea to overcome this difficulty which we now describe as follows.


Our approach is to first perform an upper triangularization of $A$ to obtain a closed differential system of the two components of $\WW^{\nc}$ which are not in the stable subspace of $A$ (we refer to these two components as {\it outgoing modes} and the other two components as {\it incoming modes}). This way the differential system \eqref{idea} is transformed to
\begin{equation}\label{UT}
\frac{d}{dx_2}\begin{pmatrix}
\WW^{in}\\
\WW^{out}
\end{pmatrix}=
\begin{pmatrix}
G & *\\
0  & H\\
\end{pmatrix}\begin{pmatrix}
\WW^{in}\\
\WW^{out}
\end{pmatrix},
\end{equation}
where $\WW^{in}$ and $\WW^{out}$ are both two-dimensional vectors corresponding to the incoming and outgoing modes of $\WW^{\nc}$, $G$ and $H$ are both $2\times2$ matrices.  
The closed differential system of $\WW^{out}$ is 
$$\frac{d}{dx_2}\WW^{out}=H\WW^{out}.$$
This upper triangularization ensures that all the eigenvalues of $H$ have positive real parts (in fact for our case, $H$ can actually be a diagonal matrix whose diagonal entries   both  have positive real parts).  
 Hence an exact estimates of $\WW^{out}$ can be obtained; 
furthermore by the $L^2$-regularity of $\WW^{\nc}$, one has that $\WW^{out}=0$ for $x_2\ge 0$. Therefore to obtain the full estimates for $\WW^{\nc}|_{x_2=0}$, it remains to estimate the other two components of $\WW^{\nc}|_{x_2=0}$, i.e. $\WW^{in}|_{x_2=0}$, which are in the stable subspace of $A$. Instead of integrating the first two rows of \eqref{UT} to derive the estimates of $\WW^{in}|_{x_2=0}$, we only use the boundary conditions in \eqref{idea}. In fact, from $\WW^{out}|_{x_2=0}=0$, the boundary conditions are reduced to
\begin{equation}
P\WW^{in}=h, \quad \text{at } x_2=0,
\end{equation}
where $P$ is a $2\times2$ matrix whose determinant is exactly the Lopatinskii determinant.  Thus, if the Lopatinskii determinant is not zero, $P$ is invertible and the matrix norm  of  $P^{-1}$ can be estimated. Therefore  $\WW^{in}|_{x_2=0} = P^{-1}h$, and  $\WW^{in}|_{x_2=0}$ is controlled by $h$. This together with the fact that $\WW^{out}|_{x_2=0} = 0$ lead to the estimates of $\WW^{\nc}|_{x_2=0}$ and hence complete the proof of stability. We further point out that our new approach can be applied to other fluid models including the Euler and MHD flows, which will be illustrated in the last section. 
\bigskip


Before we conclude this section, we introduce the following change of unknowns in order to simplify the notations in our proof of Theorem \ref{main thm}:
$$\widetilde{W}=\exp(-\gamma t)W, \quad \widetilde{\varphi}=\exp(-\gamma t)\varphi$$
 with $\gamma\geq1$. Denote two new operators $\HL^\gamma$ and $\B^\gamma$ by
\begin{equation*}
\begin{split}
&\HL^\gamma \widetilde{W} =e^{-\gamma t}\HL W  = \gamma\A_0\widetilde{W}+\A_0\p_t\widetilde{W}+\A_1\p_1\widetilde{W}+\A_2\p_2\widetilde{W} ,\\
&\B^\gamma (\widetilde{W}^{\n}, \widetilde{\varphi})=e^{-\gamma t}\B (W^{\n}, \varphi)  =\underline{M}\widetilde{W}^{\n}+\underline{b}
\begin{pmatrix}
\gamma\widetilde{\varphi}+\p_t\widetilde{\varphi}\\
\p_1\widetilde{\varphi}
\end{pmatrix}.
\end{split}
\end{equation*}
We note that $\vertiii{e^{-\gamma t}v}_{s,\gamma}$ and $\|e^{-\gamma t} u\|_{s,\gamma}$ are equivalent to the norms $\vertiii{v}_{L^2(H^s_\gamma)}$ and $\|u\|_{H^s_\gamma}$ respectively. 
Then we have the following equivalent formulation of Theorem \ref{main thm}. In the later context, we aim to prove the following theorem.

\begin{Theorem}\label{thm2}
(1) If the particular solution defined by \eqref{specialsolution} satisfies \eqref{case1}, there is a positive constant $C$ such that for all { $\gamma\geq1$}, $\widetilde{W}\in H^2(\R^3_+)$ and $\widetilde{\varphi}\in H^2(\R^3_+)$, the following estimate holds
\begin{equation}\label{est1}
\begin{split}
&\gamma\vertiii {\widetilde{W}}^2_{0,\gamma}+\left\|\widetilde{W}^{\nn}|_{x_2=0}\right\|^2_{0,\gamma}+\left\|\widetilde{\varphi}\right\|^2_{1,\gamma}\\
&\quad\quad\leq C\left(\frac{1}{\gamma^3}\vertiii{\HL^\gamma \widetilde{W}}^2_{1,\gamma}+\frac{1}{\gamma^2}\left\|\B^\gamma (\widetilde{W}^{\nn}|_{x_2=0}, \widetilde{\varphi})\right\|^2_{1,\gamma}\right).
\end{split}
\end{equation}

(2) If the particular solution defined by \eqref{specialsolution} satisfies \eqref{case2}
Then there is a positive constant $C$ such that for all { $\gamma\geq1$}, $\widetilde{W}\in H^3(\R^3_+)$ and $\widetilde{\varphi}\in H^3(\R^3_+)$, the following estimate holds
\begin{equation}\label{est2}
\begin{split}
&\gamma\vertiii {\widetilde{W}}^2_{0,\gamma}+\left\|\widetilde{W}^{\nn}|_{x_2=0}\right\|^2_{0,\gamma}+\left\|\widetilde{\varphi}\right\|^2_{1,\gamma}\\
&\quad\quad\leq C\left(\frac{1}{\gamma^5}\vertiii{\HL^\gamma \widetilde{W}}^2_{2,\gamma}+\frac{1}{\gamma^4}\left\|\B^\gamma (\widetilde{W}^{\nn}|_{x_2=0}, \widetilde{\varphi})\right\|^2_{2,\gamma}\right).
\end{split}
\end{equation}

(3) If the particular solution defined by \eqref{specialsolution} satisfies \eqref{case3}
Then there is a positive constant $C$ such that for all { $\gamma\geq1$}, $\widetilde{W}\in H^4(\R^3_+)$ and $\widetilde{\varphi}\in H^4(\R^3_+)$, the following estimate holds
\begin{equation}\label{est3}
\begin{split}
&\gamma\vertiii {\widetilde{W}}^2_{0,\gamma}+\left\|\widetilde{W}^{\nn}|_{x_2=0}\right\|^2_{0,\gamma}+\left\|\widetilde{\varphi}\right\|^2_{1,\gamma}\\
&\quad\quad\leq C\left(\frac{1}{\gamma^7}\vertiii{\HL^\gamma \widetilde{W}}^2_{3,\gamma}+\frac{1}{\gamma^6}\left\|\B^\gamma (\widetilde{W}^{\nn}|_{x_2=0}, \widetilde{\varphi})\right\|^2_{3,\gamma}\right).
\end{split}
\end{equation}

(4) If the particular solution defined by \eqref{specialsolution} satisfies
\begin{equation}
(F_{11}^r)^2+(F_{12}^r)^2 < (v^r)^2 < 2c^2+(F_{11}^r)^2+(F_{12}^r)^2,
\end{equation}
the constant vortex sheets \eqref{specialsolution} is linearly unstable, in the sense that the Lopatinskii condition is violated.

\end{Theorem}

\bigskip
\section{Decomposition of the System and Elimination of the Front}\label{sec_preliminary}

In this section, as a first step for the proof of Theorem \ref{thm2} (and hence Theorem \ref{main thm}) we perform some preliminary reductions and eliminate the front $\widetilde{\varphi}$ from the linearized system. 

Consider the following inhomogeneous differential system for $\widetilde{W}$ and $\widetilde{\varphi}$ on $\R^3_+$:
\begin{equation}\label{equationgamma}
\begin{cases}
\HL^\gamma \widetilde{W}=\widetilde{f}, &\text{if $x_2>0$},\\
\B^\gamma(\widetilde{W}^\n,\widetilde{\varphi})=\widetilde{g}, &\text{if $x_2=0$},
\end{cases}
\end{equation}
where $f$ and $g$ are given sources.

\subsection{Decomposition of the system}
Due to the linearity of the system \eqref{equationgamma}, we can decompose it into two subsystems as follows. First we consider the following problem for $V$ with the homogeneous boundary conditions:
\begin{align}\label{e32}
\begin{cases}
\HL^\gamma V = \widetilde{f}, &\text{if $x_2>0$},\\
M_1V^{\n}=0, &\text{if $x_2=0$},
\end{cases}
\end{align}
where
\begin{align*}
M_1=\begin{pmatrix}
0 & 1 & 0 & 0 & 0 & 0 & 0 & 0\\
0 & 0 & 1 & 0 & 0 & 0 & 0 & 0\\
0 & 0 & 0 & 1 & 0 & 0 & 0 & 0\\
0 & 0 & 0 & 0 & 1 & 0 & 0 & 0\\
0 & 0 & 0 & 0 & 0 & 0 & 1 & 0\\
0 & 0 & 0 & 0 & 0 & 0 & 0 & 1
\end{pmatrix}.
\end{align*}
From the classical hyperbolic theory \cite{benzoni2007multi}, the boundary condition is {strictly dissipative},
and thus \eqref{e32} admits a solution satisfying following estimates:
\begin{align*}
&\gamma\vertiii{V}_0^2\leq\frac{C}{\gamma}\vertiii{f}^2_0,\qquad{\|V^{\n}|_{x_2=0}\|^2_{k,\gamma}\leq\frac{C}{\gamma}\vertiii{f}^2_{k,\gamma}},
\end{align*}
for any integer $k \geq 0$. 

Then we consider the second problem for the difference $W=\widetilde{W}-V$. In fact $W$ satisfies the following homogeneous differential equations with inhomogeneous boundary conditions:
\begin{align}\label{w2}
\begin{cases}
\HL^\gamma W = 0, &\text{if $x_2>0$},\\
\B^\gamma(W^{\n},\varphi)=g:=\widetilde{g}-\underline{M}V^{\n}, &\text{if $x_2=0$}.
\end{cases}
\end{align}
\begin{Remark}
Here we are slightly abusing notation, since $W$ was previously defined to be the transformed perturbation of the rectilinear vortex sheets, c.f. \eqref{defn W}. From now on we will consider $W$ as a solution to \eqref{w2}. 
\end{Remark}

Multiplying the equations in system \eqref{w2} by $W$ and integrating, we have
\begin{equation*}
\gamma\vertiii{W}^2_0\leq C\|W^\nc|_{x_2=0}\|^2_0\leq C\|W^{\n}|_{x_2=0}\|^2_0.
\end{equation*}
Thus it suffices to derive the following estimate on $W$:
\begin{equation}\label{est sufficient}
\LN W^{\n}|_{x_2=0} \RN^2_0+\|\varphi\|^2_{1,\gamma}\leq\frac{C}{\gamma^{2k}}\|g\|^2_{k,\gamma}
\end{equation}
to obtain all the estimates in Theorem \ref{thm2},  where $k$ will be determined accordingly.

Now we perform the Fourier transform to system \eqref{w2} with respect to tangential variables $(t, x_1)$ and denote the variables in the frequency space by $(\delta, \eta)$. Let $\tau=\gamma+\ti\delta$.
Then $\WW$, the Fourier transform of $W$, satisfies the following system:
\begin{equation}\label{symform}
\begin{cases}
\displaystyle (\tau\A_0+\ti\eta\A_1)\WW+\A_2\frac{d\WW}{d x_2}=0, &\text{if $x_2>0$},\\
\displaystyle b(\tau,\eta)\Wphi+\underline{M}\WW^{\n}=\Wg, &\text{if $x_2=0$},
\end{cases}
\end{equation}
where
\begin{align*}
b(\tau,\eta)=\underline{b}\cdot
\begin{pmatrix}
\tau\\
\ti\eta
\end{pmatrix}=
\begin{pmatrix}
2\ti v^r\eta\\
\tau+\ti v^r\eta\\
0\\
2\ti F_{11}^r\eta\\
\ti F_{11}^r\eta\\
2\ti F_{12}^r\eta\\
\ti F_{12}^r\eta
\end{pmatrix}.
\end{align*}

To utilize the homogeneity structure of the system,  we define the hemisphere 
$$\Sigma=\{(\tau,\eta):\;|\tau|^2+(v^r)^2\eta^2=1\text{ and } \Re\tau\geq 0\}$$ 
in the whole frequency space $\Pi:=\{(\tau,\eta):\;\tau\in\C,\ \eta\in\R,\ |\tau|^2+\eta^2\neq0,\ \Re\tau\geq 0\}$. It is easily seen that $\Pi=\{s\cdot(\tau,\eta):\;s>0,\ (\tau,\eta)\in\Sigma\}$. We will carry out our argument on the hemisphere $\Sigma$ and use homogeneity property to extend it to the whole frequency space $\Pi$. 
%

\subsection{Elimination of the front}

An important observation in the vortex sheet system is that the wave front $\varphi$ is only involved in the boundary conditions. Thus we can estimate the wave front by using the ellipticity of the boundary conditions of \eqref{symform} in the sense of the following lemma.

\begin{Lemma}\label{Q}
There is a $C^\infty$ map $Q: \Pi \to GL_7(\C)$ which is homogeneous of degree $0$, such that
\begin{align*}
Q(\tau,\eta)b(\tau,\eta)=
\begin{pmatrix}
0\\
0\\
\theta(\tau,\eta)\\
0\\
0\\
0\\
0
\end{pmatrix},\quad \forall (\tau,\eta)\in \Pi,
\end{align*}
where $\theta$ is homogeneous of degree $1$ and 
\begin{align*}
\min_{(\tau,\eta)\in \Sigma} |\theta(\tau,\eta)|>0.
\end{align*}
\end{Lemma}
\begin{proof}
We only sketch the proof of the lemma. The idea is to consider the map
\begin{align*}
Q=\begin{pmatrix}
0 & 0 & 1 & 0 & 0 & 0 & 0\\
\tau+\ti v^r\eta & -2\ti v^r\eta & 0 & 0 & 0 & 0 & 0\\
-2\ti v^r\eta & \bar{\tau}-\ti v^r \eta & 0 & -2\ti F_{11}^r\eta & -\ti F_{11}^r\eta & -2\ti F_{12}^r\eta & -\ti F_{12}^r\eta\\
-F_{11}^r & 0 & 0 & v^r & 0 & 0 & 0\\
-F_{11}^r & 0 & 0 & 0 & 2v^r & 0 & 0\\
-F_{12}^r & 0 & 0 & 0 & 0 & v^r & 0\\
-F_{12}^r & 0 & 0 & 0 & 0 & 0 & 2v^r\\
\end{pmatrix}
\end{align*}
on $\Sigma$ and extend it by homogeneity of degree 0 to the whole frequency space.  Then the lemma can be proved by a direct computation. 
\end{proof}

Now we multiply the boundary conditions in \eqref{symform} from the left by $Q$ and obtain the new boundary conditions:
\begin{equation}\label{newboundary}
Qb\Wphi+Q\underline{M}\WW^{\n}=Q\Wg, \quad \text{at }\ x_2 =0.
\end{equation}

With this choice of $Q$, we have $Q\underline{M}=$
{
\begin{align*}
&{\small\begin{pmatrix}
-1 & 1 & 0 & 0 & 1 & -1 & 0 & 0\\
-c(\tau-\ti v^r\eta) & -c(\tau-\ti v^r\eta) & 0 & 0 & c(\tau+\ti v^r\eta) & c(\tau+\ti v^r\eta) & 0 & 0\\
-c(\bar{\tau}-3\ti v^r\eta) & -c(\bar{\tau}-3\ti v^r\eta) & 3\ti F_{11}^r\eta & 3\ti F_{12}^r\eta & -2c\ti v^r\eta & -2c\ti v^r\eta & -2\ti F_{11}^r\eta & -2\ti F_{12}^r\eta\\
cF_{11}^r & cF_{11}^r & -v^r & 0 & -cF_{11}^r & -cF_{11}^r & v^r & 0\\
cF_{11}^r & cF_{11}^r & -2v^r & 0 & -cF_{11}^r & -cF_{11}^r & 0 & 0\\
cF_{12}^r & cF_{12}^r & 0 & -v^r & -cF_{12}^r & -cF_{12}^r & 0 & v^r\\
cF_{12}^r & cF_{12}^r & 0 & -2v^r & -cF_{12}^r & -cF_{12}^r & 0 & 0\\
\end{pmatrix}}
\end{align*}}
on $\Sigma$ which is homogeneous of degree 0. Denote the third row of the above matrix by $\ell$. Hence the third equation of the new boundary conditions \eqref{newboundary} is
\begin{align*}
\theta(\tau,\eta)\Wphi+\ell(\theta,\eta)\WW^{\n}|_{x_2=0}=b^*(\theta,\eta)\Wg, \quad \text{on $\Sigma$.}
\end{align*}
Notice that $\ell$ and $b^*$ are homogeneous of degree $0$. From the compactness of $\Sigma$ and continuity of $\ell$ and $b^*$, we know that they are bounded on $\Pi$. By Lemma \ref{Q} and a direct integration of the above equation, we have 
\begin{equation}\label{est varphi}
\|\varphi\|^2_{1,\gamma}\leq C(\|\WW^{\n}|_{x_2=0}\|^2_0+\|g\|^2_0).
\end{equation}
Moreover, the last four equations in the new boundary conditions \eqref{newboundary} are
\begin{align*}
\begin{pmatrix}
cF_{11}^r & cF_{11}^r & -v^r & 0 & -cF_{11}^r & -cF_{11}^r & v^r & 0\\
cF_{11}^r & cF_{11}^r & -2v^r & 0 & -cF_{11}^r & -cF_{11}^r & 0 & 0\\
cF_{12}^r & cF_{12}^r & 0 & -v^r & -cF_{12}^r & -cF_{12}^r & 0 & v^r\\
cF_{12}^r & cF_{12}^r & 0 & -2v^r & -cF_{12}^r & -cF_{12}^r & 0 & 0
\end{pmatrix}
\WW^{\n}|_{x_2=0}=\tilde{Q}\Wg,
\end{align*}
for all $(\tau,\eta)\in\Pi$, where $\tilde{Q}$ is the matrix consisting of the last four rows of $Q$. Isolating the third, forth, seventh and eighth columns of the left matrix in the above equations, we can obtain
\begin{align*}
\begin{pmatrix}
 -v^r   & 0       &  v^r & 0\\
 -2v^r & 0       & 0     & 0\\
 0       & -v^r   & 0     & v^r\\
 0       & -2v^r & 0     & 0
\end{pmatrix}
\left.\begin{pmatrix}
\WW_5\\
\WW_7\\
\WW_{12}\\
\WW_{14}
\end{pmatrix}\right|_{x_2=0}=\tilde{Q}\Wg-
c\begin{pmatrix}
F_{11}^r & F_{11}^r & -F_{11}^r & -F_{11}^r\\
F_{11}^r & F_{11}^r & -F_{11}^r & -F_{11}^r\\
F_{12}^r & F_{12}^r & -F_{12}^r & -F_{12}^r\\
F_{12}^r & F_{12}^r & -F_{12}^r & -F_{12}^r 
\end{pmatrix}\WW^{\nc}|_{x_2=0},
\end{align*}
on $\Pi$. Since now the left matrix above is constant and invertible, we can obtain
\begin{equation}\label{est normal}
\left(\WW_5^2+\WW_7^2+\WW_{12}^2+\WW_{14}^2\right)|_{x_2=0}\leq C\left( \left|\WW^{\nc}|_{x_2=0}\right|^2+|\Wg|^2\right).
\end{equation}
Recall the definition of $W^\n$ in \eqref{defn Wnc}. We know from \eqref{est normal} that an estimate of $\|\WW^{\nc}|_{x_2=0}\|^2_0$ leads to the estimate of $\|\WW^{\n}|_{x_2=0}\|^2_0$, and hence the estimate of $\|\varphi\|^2_{1,\gamma}$ from \eqref{est varphi}.  Therefore by \eqref{est sufficient} we know that it is sufficient to obtain the estimate of $\|\WW^{\nc}|_{x_2=0}\|^2_0$. 

Now consider the first two rows in the new boundary conditions \eqref{newboundary} at $x_2 = 0$ and the equations of \eqref{symform} for $x_2>0$. We arrived at the following system where the wave front $\varphi$ is eliminated:
\begin{align}\label{fulleqn}
&(\tau\A_0+\ti\eta\A_1)\WW+\A_2\frac{d\WW}{d x_2}=0,\\
&\beta\WW^{\nc}|_{x_2=0}= h,
\label{betanc}
\end{align}
where $h$ is the product of the first two rows of $Q$ and $\Wg$, and 
\begin{equation}\label{matrix beta}
\beta=\begin{pmatrix}
-1 & 1 & 1 & -1\\
-c(\tau-\ti v^r\eta) & -c(\tau-\ti v^r\eta) & c(\tau+\ti v^r\eta) & c(\tau+\ti v^r\eta)
\end{pmatrix}
\end{equation}
 on $\Sigma$ and is extended to all of $\Pi$ by homogeneity of degree $0$. Note that from the homogeneity of $Q$ we have the following bounds
 \begin{equation}\label{bound h}
  \left| h \right|^2 \leq C |\hat{g}|^2,
 \end{equation}
for some positive constant $C$ which is independent of $(\tau,\eta)$. The rest of this paper is devoted to the estimate of $\|\WW^{\nc}|_{x_2=0}\|^2_0$ from the system \eqref{fulleqn}-\eqref{betanc}.
\bigskip
\section{Normal Mode Analysis}\label{sec_normal mode}

In this section, we consider the normal modes and want to separate the outgoing modes from the system microlocally. This separation gives us the exact estimates on the outgoing modes. Moreover we will show in Section \ref{sec_energy estimate} that these outgoing modes are all zero. 

\subsection{Normal modes}
To obtain an estimate of $\|\WW^{\nc}|_{x_2=0}\|^2_0$, we are led to derive a closed differential system of $\WW^{\nc}$. For this purpose, we single out the following ten algebraic equations from \eqref{fulleqn}:
\begin{align*}
&(\tau+\ti v^r\eta)\WW_1-\ti c^2\eta\WW_2+\ti c^2\eta\WW_3-\ti F_{11}^r\eta\WW_4-\ti F_{12}^r\eta\WW_6=0,\\
&-\ti F_{11}^r\eta\WW_1+(\tau+\ti v^r\eta)\WW_4=0,\\
&-\ti c F_{11}^r\eta\WW_2-\ti c F_{11}^r\eta\WW_3+(\tau+\ti v^r\eta)\WW_5=0,\\
&-\ti F_{12}^r\eta\WW_1+(\tau+\ti v^r\eta)\WW_6=0,\\
&-\ti c F_{12}^r\eta\WW_2-\ti c F_{12}^r\eta\WW_3+(\tau+\ti v^r\eta)\WW_7=0,\\
&(\tau+\ti v^l\eta)\WW_8-\ti c^2\eta\WW_9+\ti c^2\eta\WW_{10}-\ti F_{11}^l\eta\WW_{11}-\ti F_{12}^l\eta\WW_{13}=0,\\
&-\ti F_{11}^l\eta\WW_8+(\tau+\ti v^l\eta)\WW_{11}=0,\\
&-\ti c F_{11}^1\eta\WW_9-\ti c F_{11}^l\eta\WW_{10}+(\tau+\ti v^l\eta)\WW_{12}=0,\\
&-\ti F_{12}^l\eta\WW_8+(\tau+\ti v^l\eta)\WW_{13}=0,\\
&-\ti c F_{12}^l\eta\WW_9-\ti c F_{12}^l\eta\WW_{10}+(\tau+\ti v^l\eta)\WW_{14}=0.
\end{align*}
Recall definition \eqref{defn W}, we know that the above equations can be explicitly solved in terms of $\WW^\nc=(\WW_2,\WW_3,\WW_9,\WW_{10})^{\top}$ as the following:
\begin{align*}
&\WW_1=\frac{\ti c^2\eta(\tau+\ti v^r\eta)}{(\tau+\ti v^r\eta)^2+\left((F_{11}^r)^2+(F_{12}^r)^2\right)\eta^2}(\WW_2-\WW_3),\\
&\WW_4=\frac{-c^2 F_{11}^r \eta^2}{(\tau+\ti v^r\eta)^2+\left((F_{11}^r)^2+(F_{12}^r)^2\right)\eta^2}(\WW_2-\WW_3),\quad \WW_5=\frac{\ti c F_{11}^r\eta}{\tau+\ti v^r\eta}(\WW_2+\WW_3),\\
&\WW_6=\frac{-c^2 F_{12}^r \eta^2}{(\tau+\ti v^r\eta)^2+\left((F_{11}^r)^2+(F_{12}^r)^2\right)\eta^2}(\WW_2-\WW_3),
\quad \WW_7=\frac{\ti c F_{12}^r\eta}{\tau+\ti v^r\eta}(\WW_2+\WW_3),\\
&\WW_8=\frac{\ti c^2\eta(\tau+\ti v^l\eta)}{(\tau+\ti v^l\eta)^2+\left((F_{11}^l)^2+(F_{12}^l)^2\right)\eta^2}(\WW_9-\WW_{10}),\\
&\WW_{11}=\frac{-c^2 F_{11}^l \eta^2}{(\tau+\ti v^l\eta)^2+\left((F_{11}^l)^2+(F_{12}^l)^2\right)\eta^2}(\WW_9-\WW_{10}),\quad \WW_{12}=\frac{\ti c F_{11}^l\eta}{\tau+\ti v^l\eta}(\WW_9+\WW_{10}),\\
&\WW_{13}=\frac{-c^2 F_{12}^l \eta^2}{(\tau+\ti v^l\eta)^2+\left((F_{11}^l)^2+(F_{12}^l)^2\right)\eta^2}(\WW_9-\WW_{10}),\quad \WW_{14}=\frac{\ti c F_{12}^l\eta}{\tau+\ti v^l\eta}(\WW_9+\WW_{10}).
\end{align*}
Then using the differential equations in \eqref{fulleqn}, we obtain the following differential equations for $\WW^{\nc}$ only:
\begin{align}\label{WWnc}
\frac{d}{d x_2}\WW^{\nc}=A\WW^{\nc},
\end{align}
where
\begin{equation}\label{matrix A}
A=\begin{pmatrix}
n^r & -m^r & 0 & 0\\
m^r & -n^r & 0 & 0\\
0 & 0 & -n^l & m^l\\
0 & 0 & -m^l & n^l
\end{pmatrix},\\
\end{equation}
and
\begin{align*}
&n^{r,l}=\frac{2(\tau+\ti v^{r,l}\eta)^2+\left((F_{11}^{r,l})^2+(F_{12}^{r,l})^2\right)\eta^2}{2c(\tau+\ti v^{r,l}\eta)}+\frac{c}{2}\frac{(\tau+\ti v^{r,l}\eta)\eta^2}{(\tau+\ti v^{r,l}\eta)^2+\left((F_{11}^{r,l})^2+(F_{12}^{r,l})^2\right)\eta^2},\\
&m^{r,l}=\frac{c}{2}\frac{(\tau+\ti v^{r,l}\eta)\eta^2}{(\tau+\ti v^{r,l}\eta)^2+\left((F_{11}^{r,l})^2+(F_{12}^{r,l})^2\right)\eta^2}-\frac{\left((F_{11}^{r,l})^2+(F_{12}^{r,l})^2\right)\eta^2}{2c(\tau+\ti v^{r,l}\eta)}.
\end{align*}

From the classical hyperbolic theory (see, for example, \cite{benzoni2007multi}), it follows that one of the key issues in the estimates of $\|\WW^{\nc}|_{x_2=0}\|^2_0$ is to bound the components of $\WW^{\nc}$ on the stable subspace of $A$.
For this reason, we first derive the following lemma of Hersh-type \cite{hersh1963mixed} on the explicit description of the stable subspace of $A$ on $\Sigma$. The proof can be done by a direct computation, and hence we omit it.

\begin{Lemma}\label{A}
For $(\tau,\eta)\in\Sigma$ and $\mathfrak{R}\tau>0$, $A$ defined in \eqref{matrix A} admits four eigenvalues $\pm\o^r$ and $\pm\o^l$, where $\Re\o^r$ and $\Re\o^l$ are negative. Moreover, the following dispersion relations hold:
\begin{equation}\label{dispersion}
\begin{split}
(\o^r)^2={(n^r)}^2-{(m^r)}^2=\frac{1}{c^2}\left[(\tau+\ti v^r\eta)^2+\left((F_{11}^r)^2+(F_{12}^r)^2\right)\eta^2\right]+\eta^2,\\
(\o^l)^2={(n^l)}^2-{(m^l)}^2=\frac{1}{c^2}\left[(\tau+\ti v^l\eta)^2+\left((F_{11}^l)^2+(F_{12}^l)^2\right)\eta^2\right]+\eta^2.
\end{split}
\end{equation}
The eigenvectors of $\o^r$, $-\o^r$, $\o^l$ and $-\o^l$ take the following forms respectively:
\begin{equation}\label{evector}
\begin{split}
&E^r_- = (a^r,b^r,0,0)^\top,\quad
E^r_+ = (a^r,c^r,0,0)^\top,\\
&E^l_- = (0,0,b^l,a^l)^\top,\quad
E^l_+ = (0,0,c^l,a^l)^\top,
\end{split}
\end{equation}
where
\begin{align*}
a^{r,l}&=m^{r,l}\alpha^{r,l},\quad b^{r,l}=(n^{r,l}-\o^{r,l})\alpha^{r,l}, \quad c^{r,l}=(n^{r,l}+\o^{r,l})\alpha^{r,l},\\
\alpha^{r,l}&=(\tau+\ti v^{r,l}\eta)\left[(\tau+\ti v^{r,l}\eta)^2+\left((F_{11}^{r,l})^2+(F_{12}^{r,l})^2\right)\eta^2\right].
\end{align*}

Both $\o^r$ and $\o^l$ can be extended continuously to all  points $(\tau,\eta)\in\Sigma$ such that $\mathfrak{R}\tau=0$, so can $E^r_\pm$ and $E^l_\pm$. The two vectors $E^r_-$ and $E^l_-$ remain linearly independent for all $(\tau,\eta)\in\Sigma$.

\end{Lemma}
\medskip
By the definition of $A$, \eqref{dispersion} actually holds on $\Pi$. Moreover, from \eqref{evector},  we can see that $A$ can not be smoothly diagonalized near the points $(\tau,\eta)\in\Sigma$ satisfying one of the following: $m^{r,l}=0$, or $\o^{r,l}=0$, or $\tau=\pm\ti v^r\eta$, or  $\tau = \ti\left(\pm v^r \pm \sqrt{(F_{11}^r)^2+(F_{12}^l)^2}\right)\eta$, because $E^r_-$ and $E^r_+$ (or $E^l_-$ and $E^l_+$) become parallel at those points. In fact, $(\tau,\eta)\in\Sigma$ are poles of $A$ when $\tau=\pm\ti v^r\eta$ or  $\tau = \ti\left(\pm v^r \pm \sqrt{(F_{11}^r)^2+(F_{12}^l)^2}\right)\eta$. 
Therefore instead of looking for a diagonalization of $A$, we perform an upper triangularization of $A$, which we refer to as {\it separation of modes}.

\subsection{Separation of modes}\label{subsec_separation}
Since the eigenbasis of $A$ may degenerate at some points on $\Sigma$,  we need to treat $A$ microlocally. This means that for each point $(\tau_0,\eta_0)\in\Sigma$, we will separate the outgoing modes of $A$ from the system \eqref{WWnc} in some open neighborhood $\V\subset\Sigma$ of that point $(\tau_0,\eta_0)$. Here we refer to the {\it outgoing modes} of $A$ as all the components of $\WW^{\nc}$ which do not belong to the stable subspace of $A$.
By using this separation, we can show later in Section \ref{sec_energy estimate} that for every $(\tau_0,\eta_0)\in\Sigma$, these outgoing modes in fact vanish in $\V\cap\{(\tau,\eta):\; \Re\tau>0\}$.
Thus from the compactness of $\Sigma$, we can extract a finite covering $\{\V_i\}_{i=1}^{N}$ of $\Sigma$ to show that the outgoing modes of $A$ vanish in the entire $\Sigma\cap\{(\tau,\eta):\; \Re\tau>0\}$.

The following proposition guarantees that the separation of modes can always be achieved for all points on $\Sigma$.

\begin{Proposition}\label{prop1}
For $\o^{r,l}$ defined in Lemma \ref{A},  we have
\begin{align*}
(\tau+\ti v^{r,l} \eta)\o^{r,l}-c\left((\o^{r,l})^2-\eta^2\right)\neq0
\end{align*}
for all $(\tau,\eta)\in\Sigma$.
\end{Proposition}
\begin{proof}
The key idea of the  proof is to examine the signs of the real and imaginary parts of $\o^{r,l}$. For this purpose, we consider the general situation where $(x+\ti y)^2= p+\ti q$ for $x$, $y$, $p$, $q\in\R$ and $x\leq0$. By a direct computation,  we can express $x$ and $y$ in terms of $p$ and $q$ as
\begin{align}\label{root}
x=-\sqrt{\frac{p+\sqrt{p^2+q^2}}{2}},\quad y=-\text{sgn}(q)\sqrt{\frac{\sqrt{p^2+q^2}-p}{2}}
\end{align}
for all $(p,q)\in \R^2\backslash\{p<0,q=0\}$. 

Next we apply the above relations to $\o^{r,l}$ at the point $(\tau,\eta)\in\Sigma$ with $\Re\tau>0$. Let $\o^{r,l}=x^{r,l}+iy^{r,l}$ and $(\o^{r,l})^2=p^{r,l}+iq^{r,l}$, where $x^{r,l}$, $y^{r,l}$, $p^{r,l}$, $q^{r,l}\in\R$.  From the definition \eqref{dispersion} of $\o^{r,l}$ we know that $x^{r,l}\leq0$ and
\begin{align}\label{p}
&p^{r,l}=\frac{\gamma^2-(\delta+v^{r,l}\eta)^2+\left((F_{11}^{r,l})^2+(F_{12}^{r,l})^2\right)\eta^2}{c^2}+\eta^2,\\ \label{q}
&q^{r,l}=\frac{2\gamma(\delta+v^{r,l}\eta)}{c^2}.
\end{align}
With this setup it is easy to see from \eqref{root} that when $(p^{r,l},q^{r,l})\not\in\{p<0,q=0\}$ and $\delta+v^{r,l}\eta\neq0$ the sign of $y^{r,l}$ is opposite to the sign of $\delta+v^{r,l}\eta$ respectively. Here we recall the fact $\gamma=\Re\tau>0$. On the other hand, when $(p^{r,l},q^{r,l})\in\{p<0,q=0\}$, \eqref{root} fails to hold. More precisely, those points correspond exactly to the case when 
$$\gamma=0,\quad \delta+v^{r,l}\eta\neq0\quad \text{ and }\quad p^{r,l}<0.$$
Therefore they all lie on the boundary of $\Sigma$. By Lemma \ref{A}, the values of $\o^{r,l}$ at the boundary of $\Sigma$ are defined as the continuous limits of the interior values of $\o^{r,l}$.  This way, we can still determine the signs of $x^{r,l}$ and $y^{r,l}$ by continuity. 
Thus by the continuous extension of $\o^{r,l}$ along the path where the ratio of $\delta$ and $\eta$ is fixed, the sign of $y^{r,l}$ is opposite to the sign of $\delta+v^{r,l}\eta$ respectively at those exceptional points $(p^{r,l},q^{r,l})\in\{p<0,q=0\}$.  Hence we conclude that 
\begin{equation}\label{signcondition}
\text{if $\delta+v^{r,l}\eta\neq0$, the sign of $y^{r,l}$ is opposite to the sign of $\delta+v^{r,l}\eta$ respectively. }
\end{equation} 
\smallskip

Now we return to the proof of the proposition. We will only prove that $(\tau+\ti v^r \eta)\o^r-c\left((\o^r)^2-\eta^2\right)\neq0$. The case that $c\left((w^l)^2-\eta^2\right)-(\tau+\ti v^l \eta)\o^l\neq0$ can be treated similarly. Assume on the contrary that
\begin{align}\label{factor1}
(\tau+\ti v^r \eta)\o^r-c\left((\o^r)^2-\eta^2\right)=0. 
\end{align}

If $\tau+\ti v^r\eta=0$, the above equation becomes $(\o^r)^2-\eta^2=0$. Combining this with \eqref{dispersion}, we obtain $\left((F_{11}^r)^2+(F_{12}^r)^2\right)\eta^2=0$, and hence $\eta=0$, which in turn  implies $\tau=0$. This contradicts the fact that $(\tau, \eta)\in\Sigma$. 

Thus we can assume $\tau+\ti v^r\eta\neq0$. From this it follows that
\begin{align}\label{factoro}
&\o^r=\frac{c\left((\o^r)^2-\eta^2\right)}{(\tau+\ti v^r \eta)}=\frac{1}{c}\left[(\tau+\ti v^r\eta)+\frac{\left((F_{11}^r)^2+(F_{12}^r)^2\right)\eta^2}{\tau+\ti v^r \eta}\right].
\end{align}
If $\gamma$ (i.e. $\Re\tau$) is positive, it is easy to check that the real part of the right hand side of the above formula is positive, which contradicts the definition that $\Re\o^r<0$.  

Thus we only need to check the situation when $\gamma=0$. In this case, we have $\tau+\ti v^r\eta=\ti(\delta+v^r\eta)\neq0$. By (\ref{factoro}), we know that $\Re\o^r=0$, and therefore $q^r=0$ and $p^r\leq0$. We further claim that $p^r\neq0$. Otherwise if $p^r=0$, from (\ref{dispersion}) and the fact that $q^r=0$ we have $\o^r=0$. Then from (\ref{factor1}) we must have $\eta=0$. By (\ref{p}) it follows that $\delta=0$, and hence $\tau=\eta=0$, which contradicts the fact that $(\tau,\eta)\in\Sigma$.

Therefore we only need to consider $(\tau,\eta)\in\Sigma$ when $\tau+\ti v^r\eta\neq0$, $\gamma=0$ and $p^r < 0$. This immediately implies that $\delta+v^r\eta\neq0$. From \eqref{signcondition} we know that the sign of $y^r = \Im \o^r$ is opposite to that of $\delta+v^r\eta$.  However by (\ref{factoro}) and the fact that $\Re \tau = 0$,
 \begin{equation*}
 \Im\o^r=\frac{(\delta+v^r\eta)^2-\left((F_{11}^r)^2+(F_{12}^r)^2\right)\eta^2}{c(\delta+v^r \eta)}.
 \end{equation*}
Since $p^r<0$, from \eqref{p} we know that $(\delta+v^r\eta)^2-\left((F_{11}^r)^2+(F_{12}^r)^2\right)\eta^2>0$. Thus the sign of $\Im\o^r$ is the same as the sign of $\delta+v^r\eta$, which is a contradiction. Hence we conclude that $(\tau+\ti v^r \eta)\o^r-c\left((\o^r)^2-\eta^2\right)$ never vanishes in $\Sigma$, which completes the proof of the proposition.
\end{proof} 

With this proposition, we can show that $E^{r,l}_-$ do not vanish at any point on $\Sigma$. Because if $E^{r,l}=0$, we have $m^{r,l}\alpha^{r,l}=0$ and $(n^{r,l}-\o^{r,l})\alpha^{r,l}=0$. By a direct computation 
we have that $\alpha^{r,l}\neq 0$. Then $m^{r,l}\alpha^{r,l}=0$ implies that $m^{r,l}=0$. From the definition of $m^{r,l}$ it follows that
\begin{align*}
\frac{c}{2}\cdot\frac{(\tau+\ti v^{r,l}\eta)\eta^2}{(\tau+\ti v^{r,l}\eta)^2+\left((F_{11}^{r,l})^2+(F_{12}^{r,l})^2\right)\eta^2}=\frac{\left((F_{11}^{r,l})^2+(F_{12}^{r,l})^2\right)\eta^2}{2c(\tau+\ti v^{r,l}\eta)}.
\end{align*}
Together with $(n^{r,l}-\o^{r,l})\alpha^{r,l}=0$ and \eqref{dispersion}, we have
\begin{align*}
(\tau+\ti v^{r,l} \eta)\o^{r,l}-c\left((\o^{r,l})^2-\eta^2\right)=0,
\end{align*}
which contradicts Proposition \ref{prop1}. 

The non-degeneracy of $E^{r,l}_-$ allows us to construct the following transformation matrix
\begin{align*}
T=\{E^r_-,F^r,E^l_-,F^l\}
\end{align*}
in a neighborhood of $(\tau_0,\eta_0)\in \Sigma$ with
\begin{equation*}
F^r = \left\{\begin{array}{ll}
(0,1,0,0)^\top, \quad & \text{ if } m^{r}\alpha^{r}\neq0 \text{ at } (\tau_0,\eta_0), \\
(1,0,0,0)^\top, \quad & \text{ if } (n^{r}-\o^{r})\alpha^{r}\neq0 \text{ at } (\tau_0,\eta_0), 
\end{array}\right.
\end{equation*}
and similarly,
\begin{equation*}
F^l = \left\{\begin{array}{ll}
(0,0,1,0)^\top, \quad & \text{ if } m^{l}\alpha^{l}\neq0 \text{ at } (\tau_0,\eta_0), \\
(0,0,0,1)^\top, \quad & \text{ if } (n^{l}-\o^{l})\alpha^{l}\neq0 \text{ at } (\tau_0,\eta_0). 
\end{array}\right.
\end{equation*}
Obviously, from the above argument, for any point $(\tau_0,\eta_0)\in\Sigma$, there is an open neighborhood $\V$ of $(\tau_0,\eta_0)$ where $T$ is continuously invertible.  Then we can obtain
\begin{align}\label{nondiag2}
T^{-1}AT=\begin{pmatrix}
& \o^r & z^r & 0     & 0\\
& 0     & -\o^r & 0     & 0\\
& 0     & 0      & \o^l  & z^l\\
& 0     & 0      & 0     & -\o^l
\end{pmatrix}
\end{align}
on $\V$ where $A$ is given in \eqref{matrix A} and 
\begin{equation*}
z^{r,l} = \left\{\begin{array}{ll}
\displaystyle -\frac{1}{\alpha^{r,l}}, \quad & \text{ if } m^{r,l}\alpha^{r,l}\neq0 \text{ at } (\tau_0,\eta_0), \\
\displaystyle \frac{m^{r,l}}{(n^{r,l}-\o^{r,l})\alpha^{r,l}}, \quad & \text{ if } (n^{r,l}-\o^{r,l})\alpha^{r,l}\neq0 \text{ at } (\tau_0,\eta_0). 
\end{array}\right.
\end{equation*}
%

\begin{Remark}\label{rk_interior}
By the definition of $m^{r,l}$, $n^{r,l}$ and $\alpha^{r,l}$, we know that $z^{r,l}$ may not be well defined at the poles of $A$ which are all located on the boundary of $\Sigma$. However, 
the estimates of $\WW^{\nc}|_{x_2=0}$ only involves the interior points of $\Sigma$. Hence it sufficies to obtain a uniform esitamtes of $\WW^{\nc}|_{x_2=0}$ in the interior of $\Sigma$, which corresponds to $\Sigma\cap\{\Re\tau>0\}$. 
\end{Remark}

\bigskip
\section{Lopatinskii Determinant}\label{sec_Lopatinskii}

In this section, we want to estimate the components of $\WW^{\nc}|_{x_2=0}$ in the stable subspace of $A$ through the boundary conditions, which requires us to investigate the invertibility of the matrix $\beta(E^r_-,E^l_-)$. This can be done by computing $\text{det}(\beta(E^r_-,E^l_-))$, which is known as the Lopatinskii determinant. By a direct computation, we can simplify the Lopatinskii determinant to be
\begin{equation}\label{Ldet}
\begin{split}
&\Delta=\text{det}(\beta(E^r_-,E^l_-))=c^4(\tau+\ti v^r \eta)(\tau+\ti v^l \eta)\left((\tau+\ti v^r \eta)\o^r-c\left((\o^r)^2-\eta^2\right)\right)\\
&\quad\quad\quad\quad\quad\left(c\left((w^l)^2-\eta^2\right)-(\tau+\ti v^l \eta)\o^l\right)(\o^l\o^r-\eta^2)(\o^r+\o^l),
\end{split}
\end{equation}
from which we see that the Lopatinskii determinant $\Delta$ can vanish at multiple places in $\Sigma$, which indicates that one can not expect the uniform Lopatinskii condition to hold. In the following lemma we detail the root distribution of $\Delta$, which is important for later discussion.

\begin{Lemma}\label{Lopatinskii}
The roots of the Lopatinskii determinant $\Delta$ are distributed in the following ways.

{\bf Case 1.} If $v^r>\sqrt{2c^2+(F_{11}^r)^2+(F_{12}^r)^2}$, then all the roots are simple and on the boundary of $\Sigma$. The Lopatinskii condition holds. More precisely, the roots are $(\tau,\eta)\in\Sigma$ such that

 1. $\tau=\pm\ti v^r\eta$,

 2. $\tau = 0$,

 3. $\tau = \pm \ti V_1 \eta$,\\
 where $V_1^2=(v^r)^2+(F_{11}^r)^2+(F_{12}^r)^2+c^2-\sqrt{c^4+4\left((F_{11}^r)^2+(F_{12}^r)^2+c^2\right)(v^r)^2}$. 
 
{\bf Case 2.} If $0<v^r<\sqrt{(F_{11}^r)^2+(F_{12}^r)^2}$, but $v^r\neq\sqrt{\frac{\left((F_{11}^r)^2+(F_{12}^r)^2\right)\left(2c^2+(F_{11}^r)^2+(F_{12}^r)^2\right)}{4\left((F_{11}^r)^2+(F_{12}^r)^2+c^2\right)}}$, then all roots are also simple and on the boundary of $\Sigma$. The Lopatinskii condition still holds. More precisely, the roots are $(\tau,\eta)\in\Sigma$ such that

 1. $\tau=\pm\ti v^r\eta$,

 2. $\tau = \pm \ti V_1 \eta$.\\

{\bf Case 3.} If $v^r=\sqrt{\frac{\left((F_{11}^r)^2+(F_{12}^r)^2\right)\left(2c^2+(F_{11}^r)^2+(F_{12}^r)^2\right)}{4\left((F_{11}^r)^2+(F_{12}^r)^2+c^2\right)}}$, then all roots are on the boundary of $\Sigma$. The Lopatinskii condition still holds. More precisely, the roots are $(\tau,\eta)\in\Sigma$ such that

 1. $\tau=\pm\ti v^r\eta$ (double roots).\\

{\bf Case 4.} If $v^r=\sqrt{2c^2+(F_{11}^r)^2+(F_{12}^r)^2}$, then all roots are on the boundary of $\Sigma$. The Lopatinskii condition still holds. More precisely, the roots are $(\tau,\eta)\in\Sigma$ such that

 1. $\tau=\pm\ti v^r\eta$ (simple roots),

 2. $\tau = 0$ (triple root).\\

{\bf Case 5.} If $v^r=\sqrt{(F_{11}^r)^2+(F_{12}^r)^2}$, then all roots are on the boundary of $\Sigma$. The Lopatinskii condition still holds. More precisely, the roots are $(\tau,\eta)\in\Sigma$ such that

 1. $\tau=\pm\ti v^r\eta$ (simple roots),

 2. $\tau = 0$ (double root).\\

{\bf Case 6.} If $\sqrt{(F_{11}^r)^2+(F_{12}^r)^2}<v^r<\sqrt{2c^2+(F_{11}^r)^2+(F_{12}^r)^2}$, then some roots are in the interior of $\Sigma$, and hence the Lopatinskii condition fails.

Moreover for the roots above, if the degree of a root $(\tau,\eta)\in\Sigma$, with $\tau=\ti \vartheta \eta$ for some real number $\vartheta$, is $k$, then we have $\Delta=(\tau-\ti\vartheta\eta)^k h(\tau,\eta)$ for some continuous $h(\tau,\eta)$ satisfying $h(\tau,\eta) \neq 0$ near the root.

\end{Lemma}

\begin{Remark}\label{remark coincide}
It is easily seen that if 
\begin{equation}
(v^r)^2=\frac{1}{4}\left((F_{11}^r)^2+(F_{12}^r)^2+c^2\right)
\end{equation}
with $c\leq \sqrt{3\left((F_{11}^r)^2+(F_{12}^r)^2\right)}$ (this corresponds to the second case of the Lemma \ref{Lopatinskii}), then one of the non-differentiable points of $A$, namely $\tau=\ti\left(-v^r+\sqrt{(F_{11}^r)^2+(F_{12}^l)^2+c^2}\right)\eta$, coincides with the root $\tau=\ti v^r\eta$ of the Lopatinskii determinant. As we have pointed out in the Introduction, this is a new phenomenon which does not appear in the compressible Euler flow.
\end{Remark}

\begin{proof}
The proof of the above lemma depends on a detailed analysis on each factor of the Lopatinskii determinant. Firstly, the factors $(\tau+\ti v^{r,l} \eta)\o^{r,l}-c\left((\o^{r,l})^2-\eta^2\right)$ are exactly the expression in Proposition \ref{prop1}. Thus we know they are never zero.

Secondly, we consider the factors $\tau+\ti v^{r,l}\eta$. Obviously $\tau=-\ti v^{r,l}\eta$ are the only simple roots to  $\tau+\ti v^{r,l}\eta=0$, respectively. 

Thirdly, we assume that
\begin{align}\label{factor2}
\o^r\o^l-\eta^2=0.
\end{align}
If $\eta=0$, we have $\o^r=\o^l=-\frac{\tau}{c}$ ,which means $\o^r\o^l\neq0$, and hence $\o^r\o^l-\eta^2\neq0$. Moreover, if $\delta+v^{r,l}\eta=0$, for example $\delta+v^r\eta=0$, then from (\ref{dispersion}) $\o^r$ is real and negative. However, since $\eta\neq0$, we know that $\delta+v^l\eta\neq0$, which implies $\Im\o^l\neq0$. Thus $\o^r\o^l$ can not be a real number, which violates (\ref{factor2}). Therefore \eqref{factor2} can not happen for $\eta=0$ or $\delta+v^{r,l}\eta=0$.

This leads us to focus only on the points where $\eta\neq0$ and $\delta+v^{r,l}\eta\neq0$. Introduce the following two variables,
\begin{align}\label{VO}
V=\frac{\tau}{\ti\eta},\quad \O^{r,l}=\frac{\o^{r,l}}{\ti\eta}.
\end{align}
From \eqref{factor2} we have $\O^r\O^l=-1$, and hence $(\O^r)^2(\O^l)^2=1$. By the (\ref{dispersion}) we know
\begin{align}\label{Or}
(\O^r)^2=\frac{1}{c^2}[(V+v^r)^2-(F_{11}^r)^2-(F_{12}^r)^2]-1,\\ \label{Ol}
(\O^l)^2=\frac{1}{c^2}[(V+v^l)^2-(F_{11}^l)^2-(F_{12}^l)^2]-1.
\end{align}
Hence we have 
\begin{align}
[(V+v^r)^2-(F_{11}^r)^2-(F_{12}^r)^2-c^2][(V+v^l)^2-(F_{11}^l)^2-(F_{12}^l)^2-c^2]=c^4,\nonumber
\end{align}
which leads to the following equation for $V^2$: 
\begin{align}
&V^4-2\left((v^r)^2+(F_{11}^r)^2+(F_{12}^r)^2+c^2\right)V^2+{v^r}^4-2\left((F_{11}^r)^2+(F_{12}^r)^2+c^2\right)(v^r)^2\nonumber\\
&\quad\quad+\left((F_{11}^r)^2-(F_{12}^r)^2\right)^2+2c^2\left((F_{11}^r)^2-(F_{12}^r)^2\right)=0.\nonumber
\end{align}
Using the quadratic formula, the two roots of the above equation are
\begin{align}\label{V1}
&V_1^2=(v^r)^2+(F_{11}^r)^2+(F_{12}^r)^2+c^2-\sqrt{c^4+4\left((F_{11}^r)^2+(F_{12}^r)^2+c^2\right)(v^r)^2},\\
&V_2^2=(v^r)^2+(F_{11}^r)^2+(F_{12}^r)^2+c^2+\sqrt{c^4+4\left((F_{11}^r)^2+(F_{12}^r)^2+c^2\right)(v^r)^2}.
\end{align}

We claim that the points $(\tau,\eta)\in\Sigma$ with $\tau=\pm\ti V_2\eta$ are not the roots of (\ref{factor2}). Without loss of generality, we can assume $V_2$ is positive. By a direct computation, we have
$V_2+ v^{r,l}>\sqrt{c^2+(F_{11}^r)^2+(F_{12}^r)^2}$ and $-V_2+ v^{r,l}<-\sqrt{c^2+(F_{11}^r)^2+(F_{12}^r)^2}$.  Now, by (\ref{root}), (\ref{q}) and \eqref{signcondition}, $\Im\o^{r,l}$ have opposite signs to $\delta+v^{r,l}\eta$ respectively. If $\tau=\ti V_2\eta$, we have $\gamma=\Re\tau=0$ and $\delta+v^{r,l}\eta=(V_2+ v^{r,l})\eta$. Therefore $\o^r$ and $\o^l$ are purely imaginary, and 
\begin{equation*}
\O^{r,l}=\frac{\Im\o^{r,l}}{\eta}\in\R,
\end{equation*}
from which we deduce that
\begin{equation*}
\sgn\LC\O^{r,l}\RC=-\sgn\LC\frac{\delta+v^{r,l}\eta}{\eta}\RC=-\sgn\LC V_2+ v^{r,l}\RC=-1.
\end{equation*}
Therefore $\O^r\O^l\neq-1$ and (\ref{factor2}) is not satisfied. Similarly, we can show that $(\tau,\eta)\in\Sigma$ with $\tau=-\ti V_2\eta$ are also not the roots of (\ref{factor2}).

Now we focus on $V_1^2$. Obviously, by (\ref{V1}), we have 
\begin{align*}
&V_1^2>0, \text{ if }  v^r>\sqrt{2c^2+(F_{11}^r)^2+(F_{12}^r)^2} \text{ or } 0<v^r< \sqrt{(F_{11}^r)^2+(F_{12}^r)^2};\\
&V_1^2=0, \text{ if } v^r=\sqrt{2c^2+(F_{11}^r)^2+(F_{12}^r)^2} \text{ or } v^r= \sqrt{(F_{11}^r)^2+(F_{12}^r)^2};\\
&V_1^2<0, \text{ if } \sqrt{2c^2+(F_{11}^r)^2+(F_{12}^r)^2} >v^r> \sqrt{(F_{11}^r)^2+(F_{12}^r)^2}.
\end{align*}

If $\sqrt{2c^2+(F_{11}^r)^2+(F_{12}^r)^2} >v^r> \sqrt{(F_{11}^r)^2+(F_{12}^r)^2}$, by (\ref{VO}), we obtain that $\tau=\pm\ti V_1\eta$ are real. Thus $\delta=0$, but $\eta\neq0$ and $\Re\tau\neq0$. This implies that $\delta+v^{r,l}\eta\neq0$. By (\ref{p}) and (\ref{q}), we know that $p^r=p^l$ and $q^r=-q^l\neq0$. Using (\ref{root}), we can have $x^r=x^l$ and $y^r=-y^l$, i.e. $\o^r$ is the complex conjugate of $\o^l$. Then $\o^r\o^l>0$, which implies that $\tau=\pm\ti V_1\eta$ are the roots of (\ref{factor2}). This way we are able to find a root $(\tau, \eta)$ with $\Re\tau>0$, which violates the Lopatinskii condition, and hence the vortex sheets are unstable. This proves {\bf Case 6} in the lemma.

For the rest cases we will consider $V_1^2 \geq 0$ and when taking the square root we always use the positive branch, that is $V_1 \geq 0$.

If $v^r>\sqrt{2c^2+(F_{11}^r)^2+(F_{12}^r)^2}$, we have $\tau=\pm\ti V_1\eta$ is purely imaginary. Without loss of generality, we only consider the case when $\tau=\ti V_1\eta$. Then $\Re\tau=0$, but $\delta\neq0$ and $\eta\neq0$. By a direct computation, we can obtain 
\begin{align}\label{superV1}
\LV V_1+v^{r,l}\RV>\sqrt{(F_{11}^r)^2+(F_{12}^r)^2+c^2}.
\end{align}
Thus $\delta+v^{r,l}\eta=(V_1+ v^{r,l})\eta\neq0$. By (\ref{dispersion}) and (\ref{superV1}), we have that $\LC\o^{r,l}\RC^2$ are both real and negative. This means $\o^{r,l}$ are purely imaginary, and from \eqref{signcondition} the signs of $\Im\o^{r,l}$ are opposite to those of $\delta+v^{r,l}\eta$, respectively. Hence 
\begin{equation*}
\sgn\LC\o^r\o^l\RC=-\sgn\LC(\delta+v^{r}\eta)(\delta+v^{l}\eta)\RC=-\sgn\LC(V_1+ v^r)(V_1+v^l)\eta^2\RC=1.
\end{equation*}
Therefore $\o^r\o^l>0$, and $(\tau,\eta)\in\Sigma$ with $\tau=\ti V_1\eta$ are roots of (\ref{factor2}). The other case when $\tau=-\ti V_1\eta$ can be treated exactly the same way. 

If $v^r<\sqrt{(F_{11}^r)^2+(F_{12}^r)^2}$, we also have that $\tau=\pm\ti V_1\eta$ is purely imaginary.  Similarly as before we only treat the case $\tau=\ti V_1\eta$. Then $\Re\tau=0$, but $\delta\neq0$ and $\eta\neq0$. Now, we have
\begin{align}\label{subV1}
\LV V_1+v^{r,l}\RV<\sqrt{(F_{11}^r)^2+(F_{12}^r)^2+c^2}.
\end{align}
By (\ref{dispersion}) and (\ref{subV1}), we have $\LC\o^{r,l}\RC^2$ are both real and positive. Thus $\o^{r,l}$ are both real and negative, which implies $\o^r\o^l>0$. Hence $(\tau,\eta)\in\Sigma$ with $\tau=\pm\ti V_1\eta$ are roots of (\ref{factor2}).

Then we want to show that under the condition $v^r>\sqrt{2c^2+(F_{11}^r)^2+(F_{12}^r)^2}$ or $0<v^r< \sqrt{(F_{11}^r)^2+(F_{12}^r)^2}$ the roots to (\ref{factor2}) are simple. Since \eqref{factor2} does not admit a root at $\eta=0$, the points $(\tau,\eta)\in\Sigma$ satisfying $\o^{r,l}=0$ are not the roots of $\o^r\o^l-\eta^2=0$
From (\ref{dispersion}), $\o^{r,l}$ are analytic near the points where $\o^{r,l}$ do not vanish. We can differentiate (\ref{Or}) and (\ref{Ol}) with respect to $V$ at $V=V_1$ to obtain
\begin{equation*}
\left.\frac{d\O^{r,l}}{dV}\right|_{V=V_1}=\frac{V_1+v^{r,l}}{\O^{r,l}c^2}.
\end{equation*}
Thus
\begin{align*}
\left.\frac{d(\O^r\O^l+1)}{dV}\right|_{V=V_1}=\frac{(V_1+v^r)(\O^l)^2+(V_1+v^l)(\O^r)^2}{c^2\O^r\O^l}.
\end{align*}
Plugging in (\ref{Or}) and (\ref{Ol}), we obtain
\begin{align*}
\left.\frac{d(\O^r\O^l+1)}{dV}\right|_{V=V_1}=\frac{2V_1\left(V_1^2-(v^r)^2-(F_{11}^r)^2-(F_{12}^r)^2-c^2\right)}{c^4\O^r\O^l}.
\end{align*}
Using (\ref{V1}), we have
\begin{align*}
\left.\frac{d(\O^r\O^l+1)}{dV}\right|_{V=V_1}\neq0.
\end{align*}
Hence we have proved $(\tau,\eta)\in\Sigma$ with $\tau=\pm\ti V_1\eta$ are all simple roots of (\ref{factor2}) provided $v^r>\sqrt{2c^2+(F_{11}^r)^2+(F_{12}^r)^2}$ or $0<v^r< \sqrt{(F_{11}^r)^2+(F_{12}^r)^2}$. More precisely, near $\tau=\pm\ti V_1\eta$, we have $\o^r\o^l-\eta^2=(\tau\pm\ti V_1\eta)h^{\pm}(\tau,\eta)$ for some continuous $h^{\pm}(\tau,\eta)\neq 0$ respectively.

If $v^r=\sqrt{2c^2+(F_{11}^r)^2+(F_{12}^r)^2}$, we obtain $\tau=\pm\ti V_1\eta=0$. In this case $\Re\tau=\delta=0$ but $\eta\neq0$. By (\ref{p}), we have $p^{r,l}=-c^2<0$. Together with the fact that $\Re\tau=0$, $\delta+v^{r,l}\eta=v^{r,l}\eta\neq0$ and \eqref{signcondition}, we infer that $\o^{r,l}$ are purely imaginary and $\Im\o^{r,l}$ have opposite signs to $v^{r,l}\eta$ respectively. This implies $\o^r\o^l>0$. Hence $(\tau,\eta)\in\Sigma$ with $\tau=0$ are roots of (\ref{factor2}).

If $v^r=\sqrt{(F_{11}^r)^2+(F_{12}^r)^2}$, we also obtain $\tau=\pm\ti V_1\eta=0$. Then $\Re\tau=\delta=0$ but $\eta\neq0$. By (\ref{p}) and (\ref{q}), we have $p^{r,l}=c^2>0$ and $q^{r,l}=0$. This implies $\o^{r,l}$ are both real and negative. Thus $\o^r\o^l>0$. Hence $(\tau,\eta)\in\Sigma$ with $\tau=0$ are roots of (\ref{factor2}).

Now we want to check the multiplicity of the root when $\tau=0$ under the condition $v^r=\sqrt{2c^2+(F_{11}^r)^2+(F_{12}^r)^2}$ or  $v^r= \sqrt{(F_{11}^r)^2+(F_{12}^r)^2}$. Similarly as in the previous case, we obtain the following first derivative of $\O^r\O^l+1$: 
\begin{align*}
\frac{d(\O^r\O^l+1)}{dV}=\frac{2V\left(V^2-(v^r)^2-(F_{11}^r)^2-(F_{12}^r)^2-c^2\right)}{c^4\O^r\O^l}.
\end{align*}
Further differentiation yields the following second derivative:
{\small
\begin{align*}
\frac{d^2(\O^r\O^l+1)}{(dV)^2}=\frac{6V^2-\left((v^r)^2+(F_{11}^r)^2+(F_{12}^r)^2+c^2\right)}{c^4\O^r\O^l}-\frac{\left(2V\left(V^2-(v^r)^2-(F_{11}^r)^2-(F_{12}^r)^2-c^2\right)\right)^2}{c^8(\O^r\O^l)^3}.
\end{align*} 
 }
Thus
\begin{align*}
\left.\frac{d(\O^r\O^l+1)}{dV}\right|_{V=V_1}=0, \quad \left.\frac{d^2(\O^r\O^l+1)}{(dV)^2}\right|_{V=V_1}\neq0.
\end{align*}
Hence $(\tau,\eta)\in\Sigma$ with $\tau=0$ are all double roots of (\ref{factor2}) if $v^r=\sqrt{2c^2+(F_{11}^r)^2+(F_{12}^r)^2}$ or $v^r= \sqrt{(F_{11}^r)^2+(F_{12}^r)^2}$. We can also conclude that $\o^r\o^l-\eta^2=\tau^2h(\tau,\eta)$ for some continuous $h(\tau,\eta)\neq 0$ near $\tau=0$.

Finally we are left to consider the last factor of \eqref{Ldet}
\begin{align}\label{factor3}
\o^r+\o^l=0.
\end{align} 
It is obvious from the (\ref{dispersion}) that if $\Re\tau>0$ then $\Re\o^r<0$ and $\Re\o^l<0$, and hence $\o^r+\o^l\neq0$. Thus we focus on the case when $\Re\tau=0$. By (\ref{q}), we have $q^{r,l}=0$. From (\ref{factor3}) and using the definition \eqref{VO}, we have $(\O^r)^2=(\O^l)^2$, which implies $p^r=p^l$.  Using (\ref{p}), we have 
\begin{align*}
2v^r\delta\eta=2v^l\delta\eta. 
\end{align*}
This implies $\delta\eta=0$. If $\eta=0$, we obtain $\delta=1$. Thus $p^r=p^l=-\frac{1}{c^2}<0$ and $\delta+v^{r,l}\eta=1>0$. From \eqref{signcondition}, we obtain $\Im\o^r=\Im\o^l<0$, which contradicts (\ref{factor3}).
Hence it must be $\delta=0$. Since $(\tau, \eta)\in\Sigma$, we have $\eta=\pm\frac{1}{v^r}$. In this case, 
\begin{align*}
p^r=p^l=\frac{(F_{11}^r)^2+(F_{12}^r)^2+c^2-(v^r)^2}{c^2}\eta^2.
\end{align*}
If $(F_{11}^r)^2+(F_{12}^r)^2+c^2-(v^r)^2>0$, we obtain that $\o^{r,l}$ are both real and negative, which contradicts  (\ref{factor3}). Otherwise if $(F_{11}^r)^2+(F_{12}^r)^2+c^2-(v^r)^2=0$, we have $\o^{r,l}=0$, which means $(0,\pm\frac{1}{v^r})$ are roots of (\ref{factor3}). On the other hand, this situation belongs to {\bf Case 6} which we have already concluded the emergence of the instability.
Therefore in the sequel, we only consider the case $(F_{11}^r)^2+(F_{12}^r)^2+c^2-(v^r)^2<0$. In this case $\o^{r,l}$ are purely imaginary. Since $\Re\tau=\delta=0$ and $\eta=\pm\frac{1}{v^r}$, we have $\delta+v^r\eta=-(\delta+v^l\eta)\neq0$.  Again a use of \eqref{signcondition} implies that the signs of $\Im\o^{r,l}$ are opposite to those of $\delta+v^{r,l}\eta$, which implies $\Im\o^r=-\Im\o^l$. Hence $\o^r+\o^l=0$, and $(0,\pm\frac{1}{v^r})\in\Sigma$ are roots of (\ref{factor3}). 

Next we turn to the multiplicity of these roots. From (\ref{dispersion}) we know $\o^r$ and $\o^l$ can not vanish simultaneously. Thus $\o^r+\o^l$ is analytic near these roots. From (\ref{VO}), (\ref{Or}) and (\ref{Ol}), we have 
\begin{equation*}
\frac{d\O^{r,l}}{dV}=\frac{V+v^{r,l}}{\O^{r,l}c^2}.
\end{equation*}
This implies 
\begin{align*}
\frac{d(\O^r+\O^l)}{dV}=\frac{V+v^r}{c^2\O^r}+\frac{V+v^l}{c^2\O^l}.
\end{align*}
Since in this case $\o^r=-\o^l\neq0$, we have $\O^r=-\O^l\neq0$ at $(0,\pm\frac{1}{v^r})$. At these roots, we have 
\begin{align*}
\frac{d(\O^r+\O^l)}{dV}=2\frac{v^r}{c^2\O^r}\neq0.
\end{align*}  
Hence, $(0,\pm\frac{1}{v^r})$ are all simple roots to (\ref{factor3}), if $(F_{11}^r)^2+(F_{12}^r)^2+c^2-(v^r)^2<0$.  Therefore we have $\o^r+\o^l=\tau h(\tau,\eta)$ for some continuous $h(\tau,\eta)\neq 0$ near $\tau=0$.

To summarize, we have derived all possible roots $(\tau,\eta)$ of the Lopatinskii determinant, namely,
\begin{equation}\label{rd}
\tau=-\ti v^{r,l}\eta,\quad \tau=\pm\ti V_1\eta \quad \text{or} \quad \tau=0,
\end{equation}
where we have assumed that $v^r>0$. In general, some of the roots may coincide, and we have already discussed the possibility that $V_1=0$ in the study of \eqref{factor2}. Now we are left to check whether $v^r=V_1$ when $V_1>0$. 
By a direct computation,  $v^r=V_1$ if and only if
\begin{align*}
v^r=\sqrt{\frac{\left((F_{11}^r)^2+(F_{12}^r)^2\right)\left(2c^2+(F_{11}^r)^2+(F_{12}^r)^2\right)}{4\left((F_{11}^r)^2+(F_{12}^r)^2+c^2\right)}}.
\end{align*}
Obivously, $\sqrt{\frac{\left((F_{11}^r)^2+(F_{12}^r)^2\right)\left(2c^2+(F_{11}^r)^2+(F_{12}^r)^2\right)}{4\left((F_{11}^r)^2+(F_{12}^r)^2+c^2\right)}}<\sqrt{(F_{11}^r)^2+(F_{12}^r)^2}$.

This way, we can identify the root distribution \eqref{rd} with the six cases in the lemma:
\begin{align*}
&V_1^2>0 \text{ and } V_1\neq v^r \Longrightarrow \textbf{ Case 1, Case 2},\\
&V_1^2>0 \text{ and } V_1= v^r \Longrightarrow \textbf{ Case 3},\\
&V_1^2=0 \Longrightarrow \textbf{ Case 4, Case 5},\\
&V_1^2<0 \Longrightarrow \textbf{ Case 6}.
\end{align*} 

Putting all the above together we finish the proof of the lemma.
\end{proof}

With this lemma, we can obtain the following estimates on the stable subspace of $A$ near the roots of the Lopatinskii determinant.

\begin{Lemma}\label{degree}
Let $(\tau_0,\eta_0)\in\p\Sigma$ be a root of the Lopatinskii determinant $\Delta$. Then there is a neighborhood $\V$ of $(\tau_0, \eta_0)$ which does not contain any other roots of $\Delta$ and a constant $\kappa_0$, such that for $\forall (\tau,\eta)\in\V$ and $\forall Z^-\in\R^2$

If $v^r>\sqrt{2c^2+(F_{11}^r)^2+(F_{12}^r)^2}$,
\begin{align*}
|\beta(E^r,E^l)Z^-|^2\geq\kappa_0\gamma^{2}|Z^-|^2.
\end{align*}

If $0<v^r<\sqrt{(F_{11}^r)^2+(F_{12}^r)^2}$, but $v^r\neq\sqrt{\frac{\left((F_{11}^r)^2+(F_{12}^r)^2\right)\left(2c^2+(F_{11}^r)^2+(F_{12}^r)^2\right)}{4\left((F_{11}^r)^2+(F_{12}^r)^2+c^2\right)}}$,
\begin{align*}
|\beta(E^r,E^l)Z^-|^2\geq\kappa_0\gamma^{2}|Z^-|^2.
\end{align*}

If $v^r=\sqrt{\frac{\left((F_{11}^r)^2+(F_{12}^r)^2\right)\left(2c^2+(F_{11}^r)^2+(F_{12}^r)^2\right)}{4\left((F_{11}^r)^2+(F_{12}^r)^2+c^2\right)}}<\sqrt{(F_{11}^r)^2+(F_{12}^r)^2}$,
\begin{align*}
|\beta(E^r,E^l)Z^-|^2\geq\kappa_0\gamma^{4}|Z^-|^2.
\end{align*}

If $v^r=\sqrt{2c^2+(F_{11}^r)^2+(F_{12}^r)^2}$,

when $\tau_0=\pm\ti v^r\eta_0$,
\begin{align*}
|\beta(E^r,E^l)Z^-|^2\geq\kappa_0\gamma^{2}|Z^-|^2;
\end{align*}

when $\tau_0=0$,
\begin{align*}
|\beta(E^r,E^l)Z^-|^2\geq\kappa_0\gamma^{6}|Z^-|^2.
\end{align*}

If $v^r=\sqrt{(F_{11}^r)^2+(F_{12}^r)^2}$,

when $\tau_0=\pm\ti v^r\eta_0$,
\begin{align*}
|\beta(E^r,E^l)Z^-|^2\geq\kappa_0\gamma^{2}|Z^-|^2;
\end{align*}

when $\tau_0=0$,
\begin{align*}
|\beta(E^r,E^l)Z^-|^2\geq\kappa_0\gamma^{4}|Z^-|^2.
\end{align*}
\end{Lemma}

\begin{proof}
We denote the elements in the Lopatinskii matrix by
\begin{align*}
\beta(E^r,E^l)=
\begin{pmatrix}
-a^r+b^r & a^l-b^l \\
-c(\tau-\ti v^r \eta)(a^r+b^r) & c(\tau+ \ti v^r\eta)(a^l+b^l)
\end{pmatrix}=:\begin{pmatrix}
d_{11} & d_{12} \\
 d_{21}& d_{22}
 \end{pmatrix}.
\end{align*}
Since every element of $\beta(E^,E^l)$ is continuous, we know that if there is an element of $\beta(E^r,E^l)$ which is nonzero at $(\tau_0,\eta_0)$, then there is an open neighborhood $\V$ of $(\tau_0,\eta_0)$ such that $\beta(E^,E^l)$ can be transformed to a diagonal matrix in $\V$, that is
\begin{equation}\label{trans matrix}
P \beta(E_-^r, E_-^l)Q = \begin{pmatrix}
1 & 0\\
0 & \Delta
\end{pmatrix}
\end{equation}
For some $P, Q$ continuously invertible in $\V$. For example, if $d_{11}\neq 0$, then we have the following identity
\begin{equation*}
\begin{pmatrix}
1/d_{11}          & 0\\
-d_{21}/d_{11} & 1
\end{pmatrix}\beta(E^r,E^l)
\begin{pmatrix}
1 & -d_{12}\\
0 & d_{11}
\end{pmatrix}=
\begin{pmatrix}
1 & 0\\
0 & \Delta
\end{pmatrix}.
\end{equation*} 
Now we claim that there is always an element in $\beta(E^r,E^l)$ which is not zero for all $(\tau,\eta)$ in $\Sigma$. First we consider 
\begin{align*}
&d_{11}=-\left[(\tau+\ti v^r \eta)^2+\left((F_{11}^r)^2+(F_{12}^r)^2\right)\eta^2\right]\left[(\tau+\ti v^r \eta)\o^r-c((\o^r)^2-\eta^2)\right],\\
&d_{12}=-\left[(\tau+\ti v^l \eta)^2+\left((F_{11}^l)^2+(F_{12}^l)^2\right)\eta^2\right]\left[(\tau+\ti v^l \eta)\o^l-c((\o^l)^2-\eta^2)\right].
\end{align*}
From Proposition \ref{prop1}, $(\tau+\ti v^r \eta)\o^r-c((\o^r)^2-\eta^2)$ and $(\tau+\ti v^l \eta)\o^l-c((\o^l)^2-\eta^2)$ are never zero. Thus $d_{11}$ only vanishes when $\tau=-\ti v^r\eta\pm\ti \sqrt{(F_{11}^r)^2+(F_{12}^r)^2}\eta$ , and $d_{12}$ only vanishes when $\tau=\ti v^r\eta\pm\ti \sqrt{(F_{11}^r)^2+(F_{12}^r)^2}\eta$. If $d_{11}=d_{12}=0$, we can obtain $\tau=0$ and $v^r=\sqrt{(F_{11}^r)^2+(F_{12}^r)^2}$. Since $(\tau,\eta)\in\Sigma$, we have $\eta\neq0$, which impiles that $\tau\pm\ti v^r \eta\neq0$ and $\o^{r,l}\neq0$. 
From the expression of $d_{21}$ and $d_{22}$, we know that in this case $d_{21}\neq0$ and $d_{22}\neq0$. Hence for any $(\tau_0,\eta_0)$, there is an open neighborhood $\V$ of $(\tau_0,\eta_0)$ such that $\beta(E^r,E^l)$ can always be locally continuously transformed to $\text{diag}\{1,\Delta\}$. Therefore by the continuity and boundness of $d_{ij}$, equation \eqref{trans matrix} implies that 
\begin{align*}
|\beta(E_-^r,E_-^l)Z^-|^2\geq\kappa\text{min}(1,|\Delta|^2)|Z^-|^2,
\end{align*}
in $\V$, where $\kappa > 0$ depends on $(\tau_0,\eta_0)$. From the fact that there are only finitely many roots of $\Delta$ we know that $\V$ can be chosen so that it only contains one root of $\Delta$, which is $(\tau_0, \eta_0)$. Now combining the result in Lemma \ref{Lopatinskii}, we finish the proof.
\end{proof}
\begin{Remark}\label{rk1}
For the points $(\tau_0,\eta_0)$ where the Lopatinskii determinant is not zero, by the continuity of $\Delta$, we can also obtain an open neighborhood $\V$ of $(\tau_0,\eta_0)$ in which $\Delta\neq0$ such that
\begin{align*}
|\beta(E_-^r,E_-^l)Z^-|^2\geq\kappa_0|Z^-|^2,
\end{align*}
for all $(\tau,\eta)\in\V$ and $Z^-\in \R^2$, where $\kappa_0$ is a positive constant depending on $(\tau_0,\eta_0)$.
\end{Remark}

\bigskip
\section{Energy Estimates}\label{sec_energy estimate}
Now we want to combine the argument in the previous two sections and obtain the energy estimates. For a point $(\tau_0,\eta_0)\in\Sigma$, shrinking the neighborhood if necessary, we obtain a new neighborhood $\V$ of $(\tau_0,\eta_0)$ where we have the separation of modes of $A$ (c.f. \eqref{nondiag2}) and the estimate of the Lopatinskii determinant (c.f. Lemma \ref{degree} and Remark \ref{rk1}). We call such a point $(\tau_0,\eta_0)$ a ``generating point" of $\V$. Notice that if $\Delta(\tau_0,\eta_0)\neq0$ then $\Delta\neq0$ at every point of $\V$. Repeating this process for all points on $\Sigma$ forms a covering of $\Sigma$. Then by the compactness of the $\Sigma$, there is a finite subcovering $\{\V_i\}^N_{i=1}$ of $\Sigma$ with the corresponding generating points denoted by $\{(\tau_i,\eta_i)\}_{i=1}^{N}$. Obviously this subcovering contains all the neighborhoods $\V$ of $(\tau_0,\eta_0)$ such that $\Delta(\tau_0,\eta_0)=0$. Then we can construct a partition of unity of $\Sigma$ according to this finite subcovering, i.e., we can find $\chi_i\in C^{\infty}_{c}(\V_i)$ for $i = 1, \cdots, N$ such that $\sum\limits_{i=1}^N\chi_i^2=1$ on $\Sigma$.

Now we derive an energy estimate in each conic zone $\Pi_i=\{(\tau,\eta): s\cdot(\tau,\eta)\in\V_i,\ \text{for some } s>0 \}$. In each neighborhood $\V_i$ of $(\tau_i, \eta_i)$, we denote $T_i$ the transformation matrix of the separation of modes in this neighborhood and extend $\chi_i$ and $T_i$ by homogeneity of degree $0$ to the conic zone $\Pi_i$. Then we consider
\begin{align}\label{transformation}
Z=\chi_iT_i^{-1}\WW^{\nc}
\end{align}
for all $(\tau,\eta)\in\Pi_i$. The system that $Z=(Z_1,Z_2,Z_3,Z_4)^\top$ satisfies now becomes
\begin{equation*}
{dZ \over dx_2} = \left( T^{-1}_i A T_i \right) Z.
\end{equation*}
Recalling Remark \ref{rk_interior}, we know that it suffices to obtain uniform estimates of $Z$ for $\Re\tau>0$. Therefore in the below we only consider $\Re\tau>0$.

From \eqref{nondiag2}, the second and fourth equations are
\begin{align*}
\frac{dZ_2}{dx_2}=-\o^rZ_2\\
\frac{dZ_4}{dx_2}=-\o^lZ_4
\end{align*}
for all $(\tau,\eta)\in\Pi_i$ with $\Re\tau>0$. By \eqref{dispersion} we have $\Re\o^{r,l}(\tau,\eta)<0$ provided $\Re\tau>0$. Moreover, since $\WW(\tau,\eta,\cdot)$ is in $L^2$ and $T_i^{-1}$ is continuous and bounded from above in $\Pi_i$, we know that $Z(\tau,\eta,\cdot)$ is also in $L^2$ for every $(\tau,\eta)\in\Pi_i$. Hence, the above ODEs implies that
\begin{align}\label{increasemode}
Z_2=0\quad\text{ and } \quad Z_4=0,
\end{align}  
for all $(\tau,\eta)\in\Pi_i$ with $\Re\tau>0$. 

Next for $Z_1$ and $Z_3$, from \eqref{transformation} and \eqref{increasemode} we have
\begin{align*}
\chi_i\WW^{\nc}=T_iZ=(E^r_-,E^l_-)\begin{pmatrix}
Z_1\\
Z_3
\end{pmatrix},
\end{align*}
for all $(\tau,\eta)\in\Pi_i$ with $\Re\tau>0$. Then the boundary condition becomes
\begin{align}\label{stableestimate}
\chi_ih=\chi_i\beta\WW^{\nc}|_{x_2=0}=\beta(E^r_-,E^l_-)\left.\begin{pmatrix}
Z_1\\
Z_3
\end{pmatrix}\right|_{x_2=0},
\end{align}
for all $(\tau,\eta)\in\Pi_i$ with $\Re\tau>0$. In the above equation, $\det\LC\beta(E^r_-,E^l_-)\RC$ is the Lopatinskii determinant $\Delta$. From Remark \ref{rk1}, if $\det\LC\beta(E^r_-,E^l_-)\RC$ is not zero at $(\tau_i,\eta_i)$ we have
\begin{align*}
|\beta(E^r_-,E^l_-)Z^-|^2\geq\kappa_i|Z^-|^2,
\end{align*}
  $(\tau,\eta)\in\V_i$ and $Z^-\in\R^2$, where $\kappa_i$ is a positive constant depending on $(\tau_i,\eta_i)$. Since $\beta$ is homogeneous of degree $0$, we have
\begin{align*}
|\beta(E^r_-,E^l_-)Z^-|^2\geq\kappa_i|Z^-|^2,
\end{align*}
for all $(\tau,\eta)\in\Pi_i$ and $Z^-\in\R^2$. By \eqref{stableestimate}, we have
\begin{align}\label{invertible}
\left|\left.\begin{pmatrix}
Z_1\\
Z_3
\end{pmatrix}\right|_{x_2=0}\right|^2
\leq {\chi_i^2\over \kappa_i}\left|h\right|^2
\end{align}
for all $(\tau,\eta)\in\Pi_i$ with $\Re\tau>0$. 

If $(\tau_i,\eta_i)$ is a simple root of $\Delta$, by  Lemma \ref{degree}, we have
\begin{align*}
|\beta(E^r_-,E^l_-)Z^-|^2\geq\kappa_i\gamma^2|Z^-|^2,
\end{align*}
for all $(\tau,\eta)\in\V_i$ and $Z^-\in\R^2$. Since $\beta$ is homogeneous of degree $0$, we have
\begin{align*}
\left(|\tau|^2+(v^r)^2\eta^2\right)|\beta(E^r_-,E^l_-)Z^-|^2\geq\kappa_i\gamma^2|Z^-|^2,
\end{align*}
for all $(\tau,\eta)\in\Pi_i$ and $Z^-\in\R^2$. By \eqref{stableestimate}, we have
\begin{align}\label{simpleroot}
\left|\left.\begin{pmatrix}
Z_1\\
Z_3
\end{pmatrix}\right|_{x_2=0}\right|^2
\leq\chi_i^2\frac{\left(|\tau|^2+(v^r)^2\eta^2\right)}{\kappa_i\gamma^2}\left|h\right|^2
\end{align}
for all $(\tau,\eta)\in\Pi_i$ with $\Re\tau>0$.

If $(\tau_i,\eta_i)$ is a double root of $\Delta$, from Lemma \ref{degree}, we have
\begin{align*}
|\beta(E^r_-,E^l_-)Z^-|^2\geq\kappa_i\gamma^4|Z^-|^2,
\end{align*}
for all $(\tau,\eta)\in\V_i$ and $Z^-\in\R^2$. It follows from the homogeneity of $\beta$ that
\begin{align*}
\left(|\tau|^2+(v^r)^2\eta^2\right)^2|\beta(E^r_-,E^l_-)Z^-|^2\geq\kappa_i\gamma^4|Z^-|^2,
\end{align*}
for all $(\tau,\eta)\in\Pi_i$ and $Z^-\in\R^2$. By \eqref{stableestimate}, we have
\begin{align}\label{doubleroot}
\left|\left.\begin{pmatrix}
Z_1\\
Z_3
\end{pmatrix}\right|_{x_2=0}\right|^2
\leq\chi_i^2\frac{\left(|\tau|^2+(v^r)^2\eta^2\right)^2}{\kappa_i\gamma^4}\left|h\right|^2
\end{align}
for all $(\tau,\eta)\in\Pi_i$.

Similarly if $(\tau_i,\eta_i)$ is a triple root of $\Delta$, we have
\begin{align*}
|\beta(E^r_-,E^l_-)Z^-|^2\geq\kappa_i\gamma^6|Z^-|^2,
\end{align*}
for all $(\tau,\eta)\in\V_i$ and $Z^-\in\R^2$. Using homogeneity of $\beta$, once again we have
\begin{align*}
\left(|\tau|^2+(v^r)^2\eta^2\right)^3|\beta(E^r_-,E^l_-)Z^-|^2\geq\kappa_i\gamma^6|Z^-|^2,
\end{align*}
for all $(\tau,\eta)\in\Pi_i$ and $Z^-\in\R^2$. By \eqref{stableestimate}, we have
\begin{align}\label{tripleroot}
\left|\left.\begin{pmatrix}
Z_1\\
Z_3
\end{pmatrix}\right|_{x_2=0}\right|^2
\leq\chi_i^2\frac{\left(|\tau|^2+(v^r)^2\eta^2\right)^3}{\kappa_i\gamma^6}\left|h\right|^2
\end{align}
for all $(\tau,\eta)\in\Pi_i$.

Putting together \eqref{increasemode}, \eqref{invertible}-\eqref{tripleroot} we obtain the following estimate for $Z$ in $\Pi_i$:
\begin{equation}\label{Zestimate}
\left|\left.
Z\right|_{x_2=0}\right|^2
\leq\chi_i^2\frac{\left(|\tau|^2+(v^r)^2\eta^2\right)^j}{\kappa_i\gamma^{2j}}\left|h\right|^2,
\end{equation}
where $j=0$ corresponds to the case when $\beta(E^r_-,E^l_-)$ is invertible at $(\tau_i,\eta_i)$, and $j=1, 2, 3$ corresponds to the multiplicity of $(\tau_i,\eta_i)$ as a root of $\Delta$.

\medskip

Now we proceed to the proof of Theorem \ref{thm2}. 
\begin{proof}[Proof of Theorem \ref{thm2}]
When \eqref{case1} holds, from Lemma \ref{Lopatinskii} we know
either $\beta(E^r_-,E^l_-)$ is invertible at $(\tau_i,\eta_i)$ or $(\tau_i,\eta_i)$ is a simple root of $\Delta$. Then from \eqref{Zestimate} we can always have
\begin{equation*}
\left|\left.
Z\right|_{x_2=0}\right|^2
\leq\chi_i^2\frac{\left(|\tau|^2+(v^r)^2\eta^2\right)}{\kappa_i\gamma^{2}}\left|h\right|^2,
\end{equation*}
for all $(\tau,\eta)\in \Pi_i$ with $\Re\tau>0$ and all $i=1, \cdots, N$. From \eqref{transformation}, we have
\begin{align*}
\chi_i^2\left|\left.T_{i}^{-1}\WW^{\nc}\right|_{x_2=0}\right|^2\leq\chi_i^2\frac{\left(|\tau|^2+(v^r)^2\eta^2\right)}{\kappa_i\gamma^2}\left|h\right|^2.
\end{align*}
Using the boundedness of $T_i$ in $\Pi_i$ and summing up the estimates over all the conic zones $\{\Pi_i\}_{i=1}^N$, we obtain
\begin{align*}
\left|\left.\WW^{\nc}\right|_{x_2=0}\right|^2\leq C\frac{\left(|\tau|^2+(v^r)^2\eta^2\right)}{\gamma^2}\left|h\right|^2
\end{align*}
for all $(\tau,\eta)\in\Pi$ with $\Re\tau>0$. Then integrating the above inequality with respect to $(\delta,\eta)$ over $\R^2$ and recalling \eqref{bound h} we have
\begin{align*}
\left\|\WW^{\nc}\right\|_0^2\leq\frac{C}{\gamma^2}\|g\|^2_{1,\gamma},
\end{align*}
which implies \eqref{est1}.

The other two cases can be treated the same way. For the sake of completeness we provide the details. When \eqref{case2} holds, Lemma \ref{Lopatinskii} indicates that $(\tau_i,\eta_i)$ may be a nonzero point, simple root or double root of $\Delta$. From \eqref{Zestimate} we always have
\begin{equation*}
\left|\left.
Z\right|_{x_2=0}\right|^2
\leq\chi_i^2\frac{\left(|\tau|^2+(v^r)^2\eta^2\right)^2}{\kappa_i\gamma^{4}}\left|h\right|^2,
\end{equation*}
for all $(\tau,\eta)\in \Pi_i$ with $\Re\tau>0$ and all $i=1, \cdots, N$. Thus from \eqref{transformation}, we have
\begin{align*}
\chi_i^2\left|\left.T_{i}^{-1}\WW^{\nc}\right|_{x_2=0}\right|^2\leq\chi_i^2\frac{\left(|\tau|^2+(v^r)^2\eta^2\right)^2}{\kappa_i\gamma^4}\left|h\right|^2.
\end{align*}
Patching up the estimates over all the conic zones $\{\Pi_i\}_{i=1}^N$ we obtain
\begin{align*}
\left|\WW^{\nc}\right|^2\leq C\frac{\left(|\tau|^2+(v^r)^2\eta^2\right)^2}{\gamma^4}\left|h\right|^2
\end{align*}
for all $(\tau,\eta)\in\Pi$ with $\Re\tau>0$. Integration in $(\delta,\eta)$ over $\R^2$ implies
\begin{align*}
\left\|\WW^{\nc}\right\|_0^2\leq\frac{C}{\gamma^4}\|g\|^2_{2,\gamma},
\end{align*}
which proves \eqref{est2}.

The last case when \eqref{case3} is satisfied, from Lemma \ref{Lopatinskii} we know
$(\tau_i,\eta_i)$ may be a nonzero point, simple root or triple root of $\Delta$, and hence we have
\begin{equation*}
\left|\left.
Z\right|_{x_2=0}\right|^2
\leq\chi_i^2\frac{\left(|\tau|^2+(v^r)^2\eta^2\right)^3}{\kappa_i\gamma^{6}}\left|h\right|^2,
\end{equation*}
for all $(\tau,\eta)\in \Pi_i$ with $\Re\tau>0$ and all $i=1, \cdots, N$. Converting into $\WW^\nc$, we have
\begin{align*}
\chi_i^2\left|\left.T_{i}^{-1}\WW^{\nc}\right|_{x_2=0}\right|^2\leq\chi_i^2\frac{\left(|\tau|^2+(v^r)^2\eta^2\right)^3}{\kappa_i\gamma^6}\left|h\right|^2.
\end{align*}
Therefore
\begin{align*}
\left|\WW^{\nc}\right|^2\leq C\frac{\left(|\tau|^2+(v^r)^2\eta^2\right)^3}{\gamma^6}\left|h\right|^2
\end{align*}
for all $(\tau,\eta)\in\Pi$ with $\Re\tau>0$, which shows that
\begin{align*}
\left\|\WW^{\nc}\right\|_0^2\leq\frac{C}{\gamma^6}\|g\|^2_{3,\gamma},
\end{align*}
proving estimate \eqref{est3}. Thus we finish the proof of our main theorem.

\end{proof}

 \bigskip 
\section{Applications to Other Models} \label{application}

In this section, we provide two examples where our method of separation of modes can be applied.
\subsection{Euler} 
First, we consider the compressible vortex sheets in the following isentropic Euler flow (c.f. \cite{coulombel2004stability})
\begin{equation*}
\begin{split}
&\p_t\rho+\Dv(\rho\u)=0,\\
&\p_t(\rho\u)+\Dv(\rho\u\otimes\u)+\nabla p=0.
\end{split}
\end{equation*}
We use the same change of variable $\Phi$, and linearize the system around the rectilinear vortex sheets
\begin{equation*}
\begin{split}
U^r=\begin{pmatrix}
\rho\\
v^r\\
0
\end{pmatrix},
U^l=\begin{pmatrix}
\rho\\
v^l\\
0
\end{pmatrix},
\Phi^{r,l}(t, x_1, x_2)=\pm x_2,
\end{split}
\end{equation*}
where $v^r=-v^l>0$ and $\rho$ are real numbers. 
We keep the same argument as in \cite{coulombel2004stability} up to Lemma $4.2$. By \cite[Lemma 4.2]{coulombel2004stability}, we know that $\o^r$ and $\o^l$ have nonzero real parts, provided that $\Re\tau>0$ on $\Sigma$, and the corresponding eigenvectors are
\begin{equation*}
\begin{split}
&E^r_-(\tau,\eta)=\left(\frac{c}{2}\eta^2,\frac{1}{c}(\tau+\ti v^r\eta)^2+\frac{c}{2}\eta^2-(\tau+\ti v^r\eta)\o^r,0,0\right)^\top,\\
&E^l_-(\tau,\eta)=\left(0,0,\frac{1}{c}(\tau+\ti v^l\eta)^2+\frac{c}{2}\eta^2-(\tau+\ti v^l\eta)\o^l,\frac{c}{2}\eta^2\right)^\top.
\end{split}
\end{equation*} 
By a direct computation, and note that $\o^r$ and $\o^l$ are the roots with negative real parts for $\Re\tau>0$ on $\Sigma$, we have $E^r_-\neq0$ and $E^l_-\neq0$. By the same argument as in Section \ref{subsec_separation}, for $\forall (\tau_0,\eta_0)\in\Sigma$, there is an open neighborhood $\V$ and a matrix $T(\tau,\eta)$ which is invertible in $\V$, such that 
\begin{equation}\label{eu}
TAT^{-1}=\begin{pmatrix}
\o^r & z^r  & 0 & 0\\
0    & -\o^r & 0 & 0\\
0    & 0     & \o^l & z^l\\
0    & 0     & 0 & -\o^l
\end{pmatrix}.
\end{equation}
Similarly as in our model, $z^r$ and $z^l$ may blow up to infinity at some points on the boundary of $\Sigma$. But since only the points $(\tau,\eta)$ with $\Re\tau>0$ will be involved in the energy estimates, the blow-up points again do not affect the stability. 

Next we turn our attention to the Lopatinskii determinant. By \cite[Lemma 4.5]{coulombel2004stability}, for any root $(\tau_0,\eta_0)$ of the Lopatinskii determinant, there is an open neighborhood $\V$ of $(\tau_0,\eta_0)$ in $\Sigma$ and a constant $\kappa_0=\kappa_0(\tau_0,\eta_0)>0$, such that following estimates hold for all $(\tau,\eta)\in\V$:
\begin{equation}\label{el}
\forall Z^-\in\C^2,\quad \left|\beta(\tau,\eta)\left(E^r_-(\tau,\eta),E^l_-(\tau,\eta)\right)Z^-\right|^2\geq\kappa_0\gamma^2|Z^-|^2
.\end{equation}
For the points $(\tau_0,\eta_0)$ which is not a root of the Lopatinskii determinant, we can still obtain an open neighborhood $\V$ of $(\tau_0,\eta_0)$ in $\Sigma$ and a constant $\kappa_0=\kappa_0(\tau_0,\eta_0)>0$, such that the following estimates hold for all $(\tau,\eta)\in\V$:
\begin{equation}\label{ei}
\forall Z^-\in\C^2,\quad \left|\beta(\tau,\eta)\left(E^r_-(\tau,\eta),E^l_-(\tau,\eta)\right)Z^-\right|^2\geq\kappa_0|Z^-|^2
.\end{equation}
Thus by shrinking the open neighborhood $\V$ for each point $(\tau_0,\eta_0)\in\Sigma$ if necessary, we can obtain the matrix $T$ and the constant $\kappa_0$ such that \eqref{eu} and \eqref{el} or \eqref{ei} hold for all $(\tau,\eta)\in\V$. 
Hence we can apply the same argument as in Section \ref{sec_energy estimate} to recover the stability results as in \cite{coulombel2004stability}.

\subsection{MHD}
The second example we consider is the following compressible vortex sheets in the following isentropic MHD (c.f. \cite{wang2013stabilization}):
\begin{equation*}
\begin{split}
&\p_t\rho+\Dv(\rho\u)=0,\\
&\p_t(\rho\u)+\Dv(\rho\u\otimes\u-\H\otimes \H)+\nabla (p+\frac{1}{2}|\H|^2)=0,\\
&\p_t\H-\nabla\times(\u\times\H)=0.
\end{split}
\end{equation*}
We linearize the system around the rectilinear vortex sheets
\begin{equation*}
\begin{split}
U^r=\begin{pmatrix}
\rho\\
0\\
v^r_2\\
0\\
H^r_2
\end{pmatrix},
U^l=\begin{pmatrix}
\rho\\
0\\
v^l_2\\
0\\
H^l_2
\end{pmatrix},
\Phi^{r,l}(t, x_1, x_2)=\pm x_1,
\end{split}
\end{equation*}
where the components of $U^{r,l}$ are constants satisfying
\begin{equation*}
v^r_2+v^l_2=0, \quad v^r_2>0, \quad |H^r_2|=|H^l_2|.
\end{equation*}
Similarly, keeping the same argument as in \cite{wang2013stabilization} up to Lemma 4.1, we obtain
\begin{align*}
&\frac{d}{d x_2}\WW^{\nc}=A\WW^{\nc},\\
&\beta\WW^{\nc}|_{x_2=0}=\hat{h},
\end{align*}
where
\begin{equation*}
\begin{split}
&A=\begin{pmatrix}
n^r & -m^r & 0 & 0\\
m^r & -n^r & 0 & 0\\
0 & 0 & -n^l & m^l\\
0 & 0 & -m^l & n^l
\end{pmatrix},\\
&\beta(\tau,\eta)=\begin{pmatrix}
-\rho{\lambda}^2 & \rho{\lambda}^2 & \rho{\lambda}^2 & -\rho{\lambda}^2\\
\lambda(\tau+\ti v^l_2\eta) & \lambda (\tau+\ti v^l_2\eta) & -\lambda (\tau+\ti v^r_2\eta) & -\lambda (\tau+\ti v^r_2\eta)
\end{pmatrix},
\end{split}
\end{equation*}
with
\begin{equation*}
\begin{split}
&m^{r,l}=\frac{1}{\lambda}\left\{-\frac{{c_A}^2\eta^2}{2(\tau+\ti v^{r,l}_2\eta)}+\frac{{c}^4\eta^2(\tau+\ti v^{r,l}_2\eta)}{2\left({\lambda}^2(\tau+\ti v^{r,l}_2\eta)^2+{c}^2{c_A}^2\eta^2\right)}\right\},\\
&n^{r,l}=\frac{1}{\lambda}\left\{(\tau+\ti v^{r,l}_2\eta)+\frac{{c_A}^2\eta^2}{2(\tau+\ti v^{r,l}_2\eta)}+\frac{{c}^4\eta^2(\tau+\ti v^{r,l}_2\eta)}{2\left({\lambda}^2(\tau+\ti v^{r,l}_2\eta)^2+{c}^2{c_A}^2\eta^2\right)}\right\},\\
&\lambda=\sqrt{c^2+c_A^2},\quad c_A=\sqrt{|H_2^r|^2/\rho}.
\end{split}
\end{equation*}
By \cite[Lemma 4.1]{wang2013stabilization}, if $(\tau,\eta)\in\Sigma$ with $\Re\tau>0$, the eigenvalues $\o^r$ and $\o^l$ have negative real parts and the corresponding eigenvectors are
\begin{equation*}
\begin{split}
&E^r_-(\tau,\eta)=\left(\alpha^r m^r,\alpha^r(n^r-\o^r),0,0\right)^\top,\\
&E^l_-(\tau,\eta)=\left(0,0,\alpha^l(n^l-\o^l),\alpha^lm^l\right)^\top,
\end{split}
\end{equation*} 
where 
\begin{equation*}
\alpha^{r,l}=(\tau+\ti v^{r,l}\eta)\left({\lambda}^2(\tau+\ti v^{r,l}\eta)^2+{c}^2{c_A}^2\eta^2\right).
\end{equation*}
Then we want to examine whether $E^{r,l}$ can vanish on $\Sigma$. Setting $E^r_-=0$, it is obvious that $\alpha^r(n^r-m^r-\o^r)=0$, which, by \cite[Lemma 4.2]{wang2013stabilization}, it is equivalent to the condition that $\alpha^r=0$ or $(\tau+\ti v^r_2\eta)^2+c_A^2\eta=0$. Note that $m^r$ is undefined at $\alpha^r=0$.  Moreover $$\lim_{\alpha^r\to0} \alpha^rm^r\neq0.$$ 
On the other hand, a direct computation shows that when $(\tau+\ti v^r_2\eta)^2+c_A^2\eta=0$, we have $E^r_-\neq 0$. Thus $E^r_-\neq 0$ in $\Sigma$. Similarly we can show that $E^l_-\neq0$ in $\Sigma$. Hence again we can apply the argument as in Section \ref{subsec_separation} to upper triangularize $A$ in an open neighborhood $\V$ of each point $(\tau_0,\eta_0)$ in $\Sigma$. For the Lopatinskii determinant, by \cite[Lemma 5.3]{wang2013stabilization} and the continuity of the Lopatinskii determinant, we can obtain \eqref{ei} and \eqref{el} for the invertible points and roots of the Lopatinskii determinant respectively. Therefore the argument in Section \ref{sec_energy estimate} can be applied, and we can derive the result in \cite{wang2013stabilization}.

\bigskip
\section*{Acknowledgements}
R.M. Chen's research was supported in part by University of Pittsburgh CRDF grant.
J. Hu's research was supported in part by University of Pittsburgh CRDF grant and the National Science
Foundation under Grant DMS-1312800.
D. Wang's research was supported in part by the National Science
Foundation under Grant DMS-1312800.

%

\end{document}